\documentclass[final,reqno]{amsart}
\usepackage[margin=1.5in,bottom=1.25in]{geometry}	
\usepackage[pagewise]{lineno}
\usepackage{bm}

\usepackage{amsmath}	
\usepackage{amssymb}	
\usepackage{amsfonts}	
\usepackage{amsthm}	
\usepackage[foot]{amsaddr}	

\usepackage[centercolon=true]{mathtools}	
\usepackage{breqn}
\usepackage[titletoc,toc]{appendix}

\mathtoolsset{%
}

\usepackage[utf8]{inputenc}	
\usepackage[T1]{fontenc}	

\usepackage[
cal=cm,
]
{mathalfa}

\usepackage{dsfont}	
\newcommand{\BF}{\textbf}

\usepackage[light,scaled=.95]{AlegreyaSans}

\usepackage[proportional,tabular,lining,sf,mono=false]{libertine}

\usepackage{acronym}	
\newcommand{\ORC}{\textup{ORC}}

\usepackage[labelfont={bf,small},labelsep=colon,font=small]{caption}	

\usepackage{subcaption}	
\usepackage[svgnames]{xcolor}	
\colorlet{MyRed}{Crimson!60!DarkRed}
\colorlet{MyBlue}{DodgerBlue!75!black}
\colorlet{MyGreen}{DarkGreen}
\colorlet{MyViolet}{DarkMagenta}

\colorlet{MyLightBlue}{DodgerBlue!20}
\colorlet{MyLightGreen}{MyGreen!20}

\colorlet{PrimalColor}{MyBlue}
\colorlet{PrimalFill}{MyLightBlue}
\colorlet{DualColor}{MyRed}

\colorlet{AlertColor}{MyRed}	
\colorlet{BadColor}{MyRed}	
\colorlet{GoodColor}{MyGreen}	
\colorlet{LinkColor}{MediumBlue}	
\colorlet{MacroColor}{MyViolet}
\colorlet{RevColor}{MediumBlue}	

\usepackage{latexsym}	
\usepackage{fontawesome}	
\usepackage{pifont}	

\usepackage{tikz}	
\usetikzlibrary{calc,patterns}	

\usepackage{array}	
\usepackage{booktabs}	
\usepackage[inline,shortlabels]{enumitem}	
\setlist[1]{topsep=\smallskipamount,itemsep=\smallskipamount,left=\parindent}
\setlist[2]{left=0pt}

\usepackage[kerning=true]{microtype}	
\usepackage{ifluatex}
\ifluatex
  \usepackage{microtype}
  \DisableLigatures{encoding = *, family = *}
\else
  \usepackage{microtype}
\fi

\usepackage{tabto}	
\usepackage{xspace}	

\usepackage[sort&compress]{natbib}	

\bibpunct[, ]{[}{]}{,}{}{,}{,}
\setcitestyle{numbers,square}

\usepackage{hyperref}
\hypersetup{
final,
colorlinks=true,
linktocpage=true,
pdfstartview=FitH,
breaklinks=true,
pdfpagemode=UseNone,
pageanchor=true,
pdfpagemode=UseOutlines,
plainpages=false,
bookmarksnumbered,
bookmarksopen=false,
bookmarksopenlevel=1,
hypertexnames=true,
pdfhighlight=/O,
hyperfootnotes=false,
urlcolor=LinkColor,linkcolor=LinkColor,citecolor=LinkColor,	
pdftitle={},
pdfauthor={},
pdfsubject={},
pdfkeywords={},
pdfcreator={pdfLaTeX},
pdfproducer={LaTeX with hyperref}
}

\usepackage[sort&compress,capitalize,nameinlink]{cleveref}	
\crefname{algo}{Algorithm}{Algorithms}
\crefname{assumption}{Assumption}{Assumptions}
\crefname{case}{Case}{Cases}
\crefname{cond}{Condition}{Conditions}
\crefname{noref}{}{}
\crefname{problem}{Problem}{Problems}

\makeatletter	
\DeclareRobustCommand{\crefnosort}[1]{%
  \begingroup\@cref@sortfalse\cref{#1}\endgroup
}
\makeatother

\usepackage{algorithm}	
\usepackage{algpseudocode}	

\theoremstyle{plain}
\newtheorem{theorem}{Theorem}	
\newtheorem{corollary}{Corollary}	
\newtheorem{lemma}{Lemma}	
\newtheorem{proposition}{Proposition}	

\newtheorem*{theorem*}{Theorem}	
\newtheorem*{corollary*}{Corollary}	

\theoremstyle{definition}
\newtheorem{assumption}{Assumption}	
\newtheorem{definition}{Definition}	
\newtheorem{example}{\raisebox{\depth}{$\blacktriangleright$}~Example}	

\newtheorem*{assumption*}{Assumptions}	
\newtheorem*{definition*}{Definition}	
\newtheorem*{example*}{\raisebox{\depth}{$\blacktriangleright$}~Example}	

\theoremstyle{remark}
\newtheorem{remark}{Remark}	

\newtheorem*{remark*}{Remark}	
\newtheorem*{notation*}{Notation}	

\newcounter{proofstep}

\numberwithin{remark}{section}	
\numberwithin{example}{section}	

\usepackage[showdeletions,suppress]{color-edits}	
\newcommand{\draft}[1]{#1}	

\newcommand{\newmacro}[2]{\newcommand{#1}{\draft{#2}}}	
\newcommand{\newop}[2]{\DeclareMathOperator{#1}{\draft{#2}}}	
\newcommand{\newoplims}[2]{\DeclareMathOperator*{#1}{\draft{#2}}}	
\newcommand{\newsmartmacro}[2]{
	\NewDocumentCommand{#1}{
		E{_}{{}}
	}{
		\draft{#2_{##1}}
	}
}

\newcommand{\eps}{\varepsilon}	
\newcommand{\wilde}{\widetilde}	

\DeclarePairedDelimiterX{\setdef}[2]{\{}{\}}{#1:#2}	
\DeclarePairedDelimiterXPP{\exclude}[1]{\mathopen{}\setminus}{\{}{\}}{}{#1}

\DeclarePairedDelimiterX{\braket}[2]{\langle}{\rangle}{#1,#2}	
\DeclarePairedDelimiterX{\internalInner}[1]{\langle}{\rangle}{#1}
\usepackage{xparse}
\NewDocumentCommand \inner {s m g}{
	\IfBooleanTF{#1}
	{
	\internalInner*{#2\IfValueT{#3}{,#3}}
	}
	{
	\internalInner{#2\IfValueT{#3}{,#3}}
	}
}

\DeclarePairedDelimiterXPP{\dnorm}[1]{}{\lVert}{\rVert}{_{\ast}}{#1}	
\DeclarePairedDelimiterXPP{\onenorm}[1]{}{\lVert}{\rVert}{_{1}}{#1}	
\DeclarePairedDelimiterXPP{\twonorm}[1]{}{\lVert}{\rVert}{_{2}}{#1}	
\DeclarePairedDelimiterXPP{\supnorm}[1]{}{\lVert}{\rVert}{_{\infty}}{#1}	

\newmacro{\nat}{i}	
\newmacro{\natA}{i}	
\newmacro{\natB}{j}	
\newmacro{\natC}{k}	
\newmacro{\nats}{\mathbb{N}}	
\newmacro{\N}{\nats}	

\newmacro{\integer}{a}	
\newmacro{\intA}{a}	
\newmacro{\intB}{b}	
\newmacro{\intC}{c}	
\newmacro{\integers}{\mathbb{Z}}	
\newmacro{\Z}{\integers}	

\newmacro{\rational}{r}	
\newmacro{\ratA}{r}	
\newmacro{\ratB}{s}	
\newmacro{\ratC}{t}	
\newmacro{\rationals}{\mathbb{Q}}	
\newmacro{\Q}{\mathcal Q}	

\newmacro{\real}{x}	
\newmacro{\realA}{x}	
\newmacro{\realB}{y}	
\newmacro{\realC}{z}	
\newmacro{\reals}{\mathbb{R}}	
\newmacro{\R}{\reals}	

\newmacro{\complex}{z}	
\newmacro{\complexA}{z}	
\newmacro{\complexB}{w}	
\newmacro{\complexC}{z}	
\newmacro{\complexes}{\mathbb{C}}	
\newmacro{\C}{\mathbb{C}}	

\newoplims{\argmax}{argmax}	
\newoplims{\argmin}{argmin}	
\newoplims{\intersect}{\bigcap}	
\newoplims{\union}{\bigcup}	

\newop{\aff}{aff}	
\newop{\bd}{bd}	
\newop{\bigoh}{\mathcal{O}}	
\newop{\card}{card}	
\newop{\cl}{cl}	
\newop{\conv}{conv}	
\newop{\clconv}{\overline{conv}}	
\newop{\crit}{crit}	
\newop{\diag}{diag}	
\newop{\diam}{diam}	
\newop{\dist}{dist}	
\newop{\dom}{dom}	
\newop{\gra}{\textup{gra}\:}
\newop{\zer}{\textup{zer}}
\newop{\Fix}{\textup{Fix}}
\newop{\eig}{eig}	
\newop{\ess}{ess}	
\newop{\Hess}{Hess}	
\newop{\ind}{ind}	
\newop{\im}{im}	
\newop{\intr}{int}	
\newop{\Jac}{Jac}	
\newop{\one}{\mathds{1}}	
\newop{\proj}{proj}	
\newop{\prox}{prox}	
\newop{\rank}{rank}	
\newop{\relint}{ri}	
\newop{\sign}{sgn}	
\newop{\supp}{supp}	
\newop{\Sym}{Sym}	
\newop{\tr}{tr}	
\newop{\unif}{unif}	
\newop{\vol}{vol}	
\newcommand{\inter}{\textup{int}}
\newcommand{\id}{\mathrm{id}}
\newcommand{\Id}{\mathrm{Id}}
	
\newcommand{\Txb}{{\bT}_{\mm^p_{\bar x},\mm^d_{\bar y}}}	
\newcommand{\Tzb}{{\bT}_{\mm_{\bar Z}}}	

\newcommand{\cf}{cf.\xspace}	
\newcommand{\eg}{e.g.,\xspace}	
\newcommand{\ie}{i.e.,\xspace}	

\newcommand{\alt}[1]{#1'}	
\newcommand{\altalt}[1]{#1''}	

\newmacro{\ball}{\mathbb{B}}	
\newmacro{\clball}{\overline{\mathbb{B}}}	
\newmacro{\sphere}{\mathbb{S}}	

\newmacro{\argdot}{\kern.5pt\boldsymbol{\cdot}\kern.5pt}	
\newmacro{\ddt}{\frac{d}{dt}}	
\newmacro{\del}{\partial}	

\newmacro{\const}{c}	
\newmacro{\constA}{a}	
\newmacro{\constB}{b}	
\newmacro{\Const}{C}	

\newmacro{\param}{\theta}	
\newmacro{\params}{\Theta}	

\newmacro{\coef}{\alpha}	
\newmacro{\coefA}{\lambda}	
\newmacro{\coefB}{\mu}	
\newmacro{\coefC}{\nu}	

\newmacro{\expA}{p}	
\newmacro{\expB}{q}	
\newmacro{\expC}{r}	

\newmacro{\precs}{\eps}		
\newmacro{\precsalt}{\delta}	

\newmacro{\asympteq}{\asymp}

\newmacro{\func}{g} 

\newmacro{\realspace}{\R^{\vdim}}	
\newmacro{\vecspace}{\mathcal{X}}	
\newmacro{\dspace}{\R^{\vdim}}	
\newmacro{\subspace}{\mathcal{W}}	

\newmacro{\coord}{i}	
\newmacro{\coordA}{i}	
\newmacro{\coordB}{j}	
\newmacro{\coordC}{k}	
\newmacro{\nCoords}{d}	
\newmacro{\dims}{\nCoords}	
\newmacro{\vdim}{\nCoords}	

\newmacro{\bvec}{e}	
\newmacro{\uvec}{u}	
\newmacro{\bvecs}{\mathcal{E}}	

\newmacro{\point}{x}	
\newmacro{\pointA}{\point}	
\newmacro{\pointB}{\point'}	
\newmacro{\pointalt}{\pointB}
\newmacro{\pointC}{\point''}	
\newmacro{\pointaltalt}{\pointC}
\newmacro{\points}{\mathcal{X}}	
\newmacro{\intpoints}{\relint\points}	

\newmacro{\base}{p}	
\newmacro{\baseA}{q}	
\newmacro{\baseB}{q}	
\newmacro{\baseC}{u}	

\newmacro{\set}{\mathcal{K}}	
\newmacro{\setA}{\set}	
\newmacro{\setB}{\alt\set}	
\newmacro{\setC}{\altalt\set}	

\newmacro{\idx}{i}
\newmacro{\idxalt}{j}
\newmacro{\idxaltalt}{l}
\newmacro{\idxaltaltalt}{k}
\newmacro{\indices}{I}
\newmacro{\indicesalt}{J}

\newmacro{\closed}{\mathcal{C}}	
\newmacro{\cpt}{\mathcal{K}}	
\newmacro{\cptalt}{\alt\cpt}	
\newmacro{\nhd}{\mathcal{U}}	
\newmacro{\nhdalt}{\W}	
\newmacro{\nhdaltalt}{\V}	
\newmacro{\nbd}{\nhd}
\newmacro{\nbdalt}{\nhdalt}
\newmacro{\nbdaltalt}{\nhdaltalt}
\newmacro{\U}{\mathcal{U}}	
\newmacro{\V}{\mathcal{V}}	
\newmacro{\W}{\mathcal{W}}	
\newmacro{\open}{\mathcal{U}}	
\newmacro{\openA}{\mathcal{U}}	
\newmacro{\openB}{\mathcal{V}}	

\newmacro{\mfld}{\mathcal{M}}	
\newmacro{\gmat}{g}	
\newmacro{\gdist}{\dist_{\gmat}}	

\newmacro{\tvec}{z}	
\newmacro{\form}{\omega}	

\newmacro{\radius}{r}
\newmacro{\Radius}{R}
\newmacro{\Radiusalt}{\widetilde{\Radius}}
\newmacro{\radiusalt}{\alt\radius}
\newmacro{\margin}{\delta}
\newmacro{\marginalt}{\alt\margin}
\newmacro{\Margin}{\Delta}
\newmacro{\connectedcomp}{\mathcal{K}}
\newmacro{\plainset}{S} 
\newmacro{\interv}{A} 
\newmacro{\domain}{D}
\newmacro{\bigcpt}{\mathcal{D}}
\newsmartmacro{\bigcptalt}{\alt\bigcpt}
\newsmartmacro{\bigcptaltalt}{\alt\alt\bigcpt}

\newmacro{\cvx}{\mathcal{C}}	
\newmacro{\subd}{\partial}	

\newop{\tspace}{T}	
\newop{\tcone}{TC}	
\newop{\dcone}{\tcone^{\ast}}	
\newop{\ncone}{NC}	
\newop{\pcone}{PC}	
\newop{\hull}{\Delta}	

\newop{\minimize}{minimize}	
\newop{\Opt}{Opt}	
\newop{\Sol}{Sol}	
\newop{\gap}{Gap}	
\newop{\orcl}{\mathsf{G}}	
\newop{\err}{\mathsf{Z}}	

\newmacro{\obj}{f}	
\newmacro{\objalt}{g}	
\newmacro{\objA}{f}	
\newmacro{\objB}{g}	
\newmacro{\sobj}{F}	

\newcommand{\sol}[1][\point]{#1^{\ast}}	

\newmacro{\gvec}{g}	
\newmacro{\gbound}{G}	

\newmacro{\oper}{A}	
\newmacro{\vecfield}{v}	
\newmacro{\vbound}{V}	

\newmacro{\lips}{L}	
\newmacro{\strong}{\alpha}	
\newmacro{\smooth}{\beta}	
\newmacro{\tmplips}{L}	
\newmacro{\tmpbound}{B}	
\newmacro{\growth}{M}	
\newmacro{\regparam}{\lambda}	

\newop{\ex}{\mathbb{E}}	
\newop{\prob}{\mathbb{P}}	
\newop{\probalt}{\mathbb{Q}}	
\newop{\var}{\mathbb{V}}	
\newop{\cov}{cov}	
\newop{\simplex}{\hull}	

\DeclarePairedDelimiterXPP{\exof}[1]{\ex}{[}{]}{}{
 #1}

\DeclarePairedDelimiterXPP{\exwrt}[2]{\ex_{#1}}{[}{]}{}{
 #2}

\DeclarePairedDelimiterXPP{\probof}[1]{\prob}{(}{)}{}{
 #1}

\DeclarePairedDelimiterXPP{\probwrt}[2]{\prob_{\!#1}}{(}{)}{}{
 #2}

\DeclarePairedDelimiterXPP{\oneof}[1]{\one}{\{}{\}}{}{#1}	

\DeclarePairedDelimiterXPP{\varof}[1]{\var}{[}{]}{}{
 #1}

\DeclarePairedDelimiterXPP{\covof}[1]{\cov}{(}{)}{}{
 #1}

\newmacro{\event}{E}	
\newmacro{\eventA}{E}	
\newmacro{\eventB}{H}	

\newmacro{\seed}{\theta}	
\newmacro{\seeds}{\Theta}	
\newmacro{\pdist}{P}	
\newmacro{\history}{\mathcal{H}}	

\newmacro{\sample}{\omega}	
\newmacro{\samples}{\Omega}	
\newmacro{\sspace}{\R^{m}}	

\newmacro{\filter}{\mathcal{F}}	
\newmacro{\probspace}{(\samples,\filter,\prob)}	

\newmacro{\mean}{\mu}	
\newmacro{\sdev}{\sigma}	
\newmacro{\variance}{\sdev^{2}}	
\newmacro{\variancealt}{s^2}

\newmacro{\covmat}{\Sigma}
\newmacro{\hessmat}{H}

\newmacro{\rv}{X}
\newmacro{\trv}{X_\Radius}
\newmacro{\gaussian}{\mathcal{N}}

\newmacro{\partition}{Z}

\newmacro{\nSamples}{n}	
\newmacro{\datapoint}{\xi}

\newmacro{\beforestart}{-1}	
\newmacro{\start}{0}	
\newmacro{\afterstart}{1}	
\newmacro{\running}{\start,\afterstart,\dotsc}	

\newmacro{\run}{n}	
\newmacro{\runA}{n}	
\newmacro{\runB}{k}	
\newmacro{\runC}{\ell}	
\newmacro{\nRuns}{N}	
\newmacro{\nRunsalt}{\alt\nRuns}	
\newmacro{\runs}{\mathcal{\nRuns}}	
\newmacro{\runalt}{\runB}	

\newmacro{\tstart}{0}	
\newmacro{\timeA}{t}	
\newmacro{\timeB}{s}	
\newmacro{\timealt}{\timeB}	
\newmacro{\timealtalt}{u}
\newmacro{\timeC}{\tau}	
\newmacro{\timeD}{\lambda}	
\newmacro{\horizon}{T}	
\newmacro{\horizonalt}{S}	

\newmacro{\seq}{a}	
\newmacro{\seqA}{a}	
\newmacro{\seqB}{b}	
\newmacro{\seqC}{c}	

\newmacro{\state}{x}	
\newsmartmacro{\accstate}{x^{\step}}	
\newmacro{\stateA}{x}	
\newmacro{\stateB}{z}	
\newmacro{\statealt}{\stateB}
\newmacro{\stateC}{y}	
\newmacro{\stateD}{p}	
\newmacro{\statealtalt}{\stateC}
\newmacro{\statealtaltalt}{\stateD}

\newmacro{\cstate}{X}
\newmacro{\cstateA}{X}
\newmacro{\cstateB}{Z}
\newmacro{\cstatealt}{\cstateB}

\newmacro{\startingpoint}{\point_\start}	

\newcommand{\curr}[1][\state]{\draft{#1_{\run}}}	

\newmacro{\mat}{M}	
\newmacro{\hmat}{H}	

\newmacro{\ones}{\mathbf{1}}	
\newmacro{\eye}{I}	
\newmacro{\identity}{\eye}	

\newmacro{\eigval}{\lambda}	

\newop{\Nash}{NE}	
\newop{\CE}{CE}	
\newop{\CCE}{CCE}	
\newop{\NI}{NI}	

\newop{\brep}{br}	
\newop{\reg}{Reg}	
\newop{\preg}{\overline{Reg}}	
\newop{\val}{val}	

\newmacro{\play}{i}	
\newmacro{\playA}{i}	
\newmacro{\playB}{j}	
\newmacro{\playC}{k}	
\newmacro{\nPlayers}{N}	
\newmacro{\players}{\mathcal{\nPlayers}}	

\newmacro{\pure}{\alpha}	
\newmacro{\pureA}{\alpha}	
\newmacro{\pureB}{\beta}	
\newmacro{\pureC}{\gamma}	
\newmacro{\nPures}{A}	
\newmacro{\pures}{\mathcal{\nPures}}	

\newmacro{\strat}{x}	
\newmacro{\stratA}{x}	
\newmacro{\stratB}{\stratA'}	
\newmacro{\stratC}{\stratA''}	
\newmacro{\strats}{\mathcal{X}}	
\newmacro{\intstrats}{\strats^{\circle}}	

\newmacro{\pay}{u}	
\newmacro{\payv}{v}	
\newmacro{\payfield}{v}	
\newmacro{\loss}{\ell}	

\newmacro{\game}{\mathcal{G}}	
\newmacro{\gameFull}{\game(\players,\points,\pay)}	

\newmacro{\fingame}{\Gamma}	
\newmacro{\fingameFull}{\Gamma(\players,\pures,\pay)}	

\newmacro{\minmax}{L}	
\newmacro{\minvar}{\point_{1}}	
\newmacro{\minvarA}{\point_{1}}	
\newmacro{\minvarB}{\minvarA'}	
\newmacro{\minvars}{\points_{1}}	
\newmacro{\maxvar}{\point_{2}}	
\newmacro{\maxvarA}{\point_{2}}	
\newmacro{\maxvarB}{\maxvarA'}	
\newmacro{\maxvars}{\points_{2}}	

\newmacro{\pot}{U}	

\newmacro{\gradientbound}{16 \dims \log 6}	

\newmacro{\hreg}{h}	
\newmacro{\breg}{D}	
\newmacro{\mprox}{P}	
\newmacro{\mirror}{Q}	
\newmacro{\fench}{F}	
\newmacro{\hstr}{K}	
\newmacro{\hrange}{H}	
\newmacro{\proxdom}{\points^{\hreg}}	

\DeclarePairedDelimiterXPP{\bregof}[2]{\breg}{(}{)}{}{#1,#2}	
\DeclarePairedDelimiterXPP{\proxof}[2]{\mprox_{#1}}{(}{)}{}{#2}	

\newmacro{\dpoint}{y}	
\newmacro{\dpointA}{y}	
\newmacro{\dpointB}{\dpointA'}	
\newmacro{\dpointC}{\dpointA''}	
\newmacro{\dpoints}{\mathcal{Y}}	

\newmacro{\dstate}{Y}	
\newmacro{\dvec}{w}	

\newmacro{\zone}{\mathbb{D}}	

\newop{\Eucl}{\Pi}	
\newop{\logit}{softmax}	
\newop{\dkl}{KL}	

\newmacro{\flowmap}{\Theta}	
\DeclarePairedDelimiterXPP{\flowof}[2]{\flowmap_{#1}}{(}{)}{}{#2}	
\DeclarePairedDelimiterXPP{\dotflowof}[2]{\dot\flowmap_{#1}}{(}{)}{}{#2}	

\newmacro{\traj}{x}	
\DeclarePairedDelimiterXPP{\trajof}[1]{\traj}{(}{)}{}{#1}	
\DeclarePairedDelimiterXPP{\difftrajof}[1]{\dot\traj}{(}{)}{}{#1}	

\newcommand{\est}[1]{\hat #1}	

\newmacro{\signal}{\est\gvec}	
\newmacro{\step}{\eta}	
\newmacro{\learn}{\eta}	
\newmacro{\tempinv}{\beta}	
\newmacro{\batch}{B}
\newmacro{\batchidx}{b}

\newmacro{\efftime}{\tau}	

\newmacro{\error}{Z}	

\newmacro{\noise}{\mathsf{U}}	
\newmacro{\snoise}{\xi}	
\newmacro{\noisepar}{\sdev}	
\newmacro{\noisevar}{\variance}	
\newmacro{\aggnoise}{\mathrm{\uppercase\expandafter{\romannumeral1}}}	
\newmacro{\supnoise}{\aggnoise_{\infty}}	
\newmacro{\maxnoise}{\aggnoise^{\ast}}	

\newmacro{\bias}{b}	
\newmacro{\drift}{b}	
\newmacro{\bbound}{B}	
\newmacro{\sbias}{\chi}	
\newmacro{\aggbias}{\mathrm{\uppercase\expandafter{\romannumeral2}}}	
\newmacro{\supbias}{\aggbias_{\infty}}	
\newmacro{\maxbias}{\aggbias^{\ast}}	

\newmacro{\sbound}{M}	
\newmacro{\aggsecond}{\mathrm{\uppercase\expandafter{\romannumeral3}}}	
\newmacro{\supsecond}{\aggsecond_{\infty}}	
\newmacro{\maxsecond}{\aggsecond^{\ast}}	

\newmacro{\mix}{\delta}	
\newmacro{\unitvec}{w}	
\newmacro{\unitvar}{W}	
\newmacro{\perturb}{z}	

\newmacro{\purequery}{\est\pure}	
\newmacro{\query}{\est\state}	
\newmacro{\pivot}{\point}	
\newmacro{\querypoint}{\est\point}	

\newmacro{\vertex}{v}	
\newmacro{\vertexA}{v}	
\newmacro{\vertexB}{w}	
\newmacro{\vertexC}{u}	
\newmacro{\nVertices}{V}	
\newmacro{\vertices}{\mathcal{V}}	

\newmacro{\edge}{e}	
\newmacro{\edgeA}{e}	
\newmacro{\edgeB}{\edgeA'}	
\newmacro{\edgeC}{\edgeA''}	
\newmacro{\nEdges}{E}	
\newmacro{\edges}{\mathcal{\nEdges}}	

\newmacro{\graph}{\mathcal{G}}	
\newmacro{\graphFull}{\graph(\vertices,\edges)}	
\newmacro{\adjmat}{A}	
\newmacro{\wmat}{W}	

\newmacro{\tree}{T}
\newmacro{\treealt}{\alt\tree}
\newmacro{\trees}{\mathcal{T}}
\newmacro{\treesalt}{\wilde\trees}

\newmacro{\mgf}{M}	
\DeclarePairedDelimiterXPP{\mgfof}[2]{\mgf_{#1}}{(}{)}{}{#2}	

\newmacro{\cgf}{K}	
\DeclarePairedDelimiterXPP{\cgfof}[2]{\cgf_{#1}}{(}{)}{}{#2}	

\newmacro{\ham}{\mathcal{H}}	
\DeclarePairedDelimiterXPP{\hamof}[2]{\ham_{#1}}{(}{)}{}{#2}	

\newmacro{\lag}{\mathcal{L}}	
\DeclarePairedDelimiterXPP{\lagof}[2]{\lag_{#1}}{(}{)}{}{#2}	

\newmacro{\mom}{p}
\newmacro{\pos}{q}
\newmacro{\vel}{v}
\newmacro{\velalt}{w}

\newmacro{\hamilt}{\mathcal{H}}
\newmacro{\hamiltalt}{\bar\hamilt}

\newmacro{\lagrangian}{\mathcal{L}}
\newmacro{\lagrangianalt}{\bar\lagrangian}

\newmacro{\curve}{\gamma}	
\DeclarePairedDelimiterXPP{\curveat}[1]{\curve}{(}{)}{}{#1}	
\DeclarePairedDelimiterXPP{\diffcurveat}[1]{\dot\curve}{(}{)}{}{#1}	

\newmacro{\curves}{\Gamma}
\DeclarePairedDelimiterXPP{\curvesat}[3]{\curves_{#1}}{(}{)}{}{#2;#3}	

\newmacro{\contcurves}{\contfuncs}
\DeclarePairedDelimiterXPP{\contcurvesat}[2]{\contcurves_{#1}}{(}{)}{}{#2}	

\newmacro{\lint}{\cstate}	
\DeclarePairedDelimiterXPP{\lintat}[1]{\lint}{(}{)}{}{#1}	

\newmacro{\qpot}{B}	
\newmacro{\qmat}{B}	
\newmacro{\energy}{E}	

\newmacro{\action}{\mathcal{S}}
\newmacro{\act}{\mathcal{S}}
\DeclarePairedDelimiterXPP{\actof}[2]{\act_{#1}}{[}{]}{}{#2}	

\newmacro{\pth}{\curve}
\newmacro{\pthalt}{\varphi}
\newmacro{\pths}{\curves}
\newmacro{\dpth}{\xi}
\newmacro{\dpthalt}{\zeta}
\newmacro{\daction}{\mathcal{A}}
\newmacro{\quasipot}{V}
\newmacro{\dquasipot}{B}
\newmacro{\symdquasipot}{C}
\newmacro{\dquasipotalt}{\widetilde{\dquasipot}}
\newmacro{\invpot}{W}
\newmacro{\dinvpot}{\energy}
\newmacro{\logmgf}{H}
\newmacro{\contfuncs}{\mathcal{C}}
\newmacro{\map}{F}
\newmacro{\rate}{\rho}
\newmacro{\level}{s}

\newmacro{\eqcl}{\mathcal{K}}
\newmacro{\neqcl}{\nComps}
\newmacro{\primvar}{\alpha}
\newmacro{\bdprimvar}{\primvar_\infty}
\newmacro{\bdvar}{\noisepar_{\infty}^{2}}
\newmacro{\exponent}{s}
\newmacro{\bdpot}{\pot_\infty}
\newmacro{\potgbound}{C}

\newmacro{\ratenhd}{\mathcal{N}}

\newmacro{\diffcurr}{\delta \curr}

\newmacro{\comp}{\mathcal{K}}
\newmacro{\iComp}{i}
\newmacro{\jComp}{j}
\newmacro{\kComp}{k}
\newmacro{\nComps}{K}

\newmacro{\iGround}{0}
\newmacro{\ground}{\comp_{\iGround}}
\newmacro{\groundstates}{\sol[\indices]}
\newmacro{\asymptstables}{AS}
\newmacro{\nontrivialattract}{NTA}

\addauthor[Pan]{PM}{MediumBlue}

\newmacro{\meas}{\mu}	
\newmacro{\altmeas}{\nu}	
\newmacro{\occmeas}{\meas}	
\newmacro{\borel}{\mathcal{B}}	
\newmacro{\size}{\delta}	
\newmacro{\toler}{\eps}	

\newcommand{\Pro}{\mathcal{P}(\Xi)}
\newcommand{\Probm}{\mathcal{P}_{1}(\Xi)}

\newcommand{\Ex}{\mathds{E}}
\newcommand{\Wass}{\mathds{W}}

\newcommand{\Lip}{\textup{Lip}}
\newcommand{\kk}{\textup{K}}
\newcommand{\E}{\mathcal E}
\newcommand{\Aa}{\mathcal A}
\newcommand{\dd}{\mathrm{d}}
\newcommand{\mm}{\mathsf{m}}
\newcommand{\hh}{\mathsf{h}}

\newcommand{\g}{\mathsf{g}}
\newcommand{\F}{\mathsf{f}}
\newcommand{\rr}{\mathsf{r}}
\newcommand{\e}{\mathbf{e}}
\newcommand{\bA}{\mathbf{A}}
\newcommand{\bB}{\mathbf{B}}
\newcommand{\bL}{\mathbf{L}}
\newcommand{\bT}{\mathbf{T}}

\newcommand{\X}{\mathcal{X}}

\newcommand{\EE}{\mathbf{E}}

\newcommand{\cB}{{\mathcal B}}

\newcommand{\cH}{{\mathcal H}}
\newcommand{\cK}{{\mathcal K}}
\newcommand{\cP}{{\mathcal P}}
\newcommand{\cZ}{{\mathcal Z}}
\newcommand{\I}{\iota}

\begin{document}

 

\title[Stochastic Monotone Inclusion]{Stochastic Monotone Inclusion with Closed Loop Distributions}

\author
[H.~Ennaji]
{Hamza Ennaji$^{c,\ast}$}
\address{$^{c}$\,%
Corresponding author.}
\address{$^{\ast}$\,%
Univ. Grenoble Alpes, CNRS, Grenoble INP*, LJK, 38000 Grenoble, France.}
\email[Corresponding author]{hamza.ennaji@univ-grenoble-alpes.fr}
\author
[J.~Fadili]
{Jalal Fadili$^{\sharp}$}
\address{$^{\sharp}$\,%
ENSICAEN, Normandie Université, CNRS, GREYC, France.}
\email{jalal.fadili@ensicaen.fr}

\author
[H.Attouch]
{Hedy Attouch$^{\diamond}$}
\address{$^{\diamond}$\,%
IMAG, Université Montpellier, CNRS, France.}
\email{hedy.attouch@umontpellier.fr}





\makeatletter	
\newcommand{\thmtag}[1]{	
  \let\oldthetheorem\thetheorem	
  \renewcommand{\thetheorem}{#1}	
  \g@addto@macro\endtheorem{	
    \addtocounter{theorem}{0}	
    \global\let\thetheorem\oldthetheorem}	
  }
\makeatother

\makeatletter	
\newcommand{\asmtag}[1]{	
  \let\oldtheassumption\theassumption	
  \renewcommand{\theassumption}{#1}	
  \g@addto@macro\endassumption{	
    \addtocounter{assumption}{0}	
    \global\let\theassumption\oldtheassumption}	
  }
\makeatother

\begin{abstract}
In this paper, we study in a Hilbertian setting, first and second-order monotone inclusions related to stochastic optimization problems with decision-dependent distributions. The studied dynamics are formulated as monotone inclusions governed by Lipschitz perturbations of maximally monotone operators  where  the concept of equilibrium plays a central role. We discuss the relationship between the $\Wass_1$-Wasserstein Lipschitz behavior of the distribution and the so-called coarse Ricci curvature. As an application, we consider the monotone inclusions associated with stochastic optimisation problems involving the sum of a smooth function with Lipschitz gradient, a proximable function and a composite term.
\end{abstract}

\allowdisplaybreaks	
\acresetall	
\acused{iid}
\acused{LHS}
\acused{RHS}

\dedicatory{Dedicated to the memory of Hedy Attouch, outstanding mathematician and beloved collaborator.}
\maketitle

\section{Introduction}
Recently, many problems in machine learning and risk management come in form of stochastic optimisation problems. Such problems aim to learn a decision rule from a data sample. This can be formulated in terms of optimization problems of the form
\begin{equation}\label{eq:optim_intro}
	\min_{x \in \R^n}~\Ex_{\xi\sim \mm}(f(x,\xi)) + g(x),
\end{equation} 
where $\mm$ is a probability measure, $\Ex_{\xi \sim \mm}$ is the expectation operator with respect to the measure $\mm$, $f(x,\xi)$ is a loss function of the decision $x$ at data point $\xi$ and $g$ is a regularizer. In this work, we are interested in \eqref{eq:optim_intro} in the case where the distribution $\mm$ depends itself on the decision $x$, \ie problems of the form	
\begin{equation}\label{eq:optim_intro_x}
	\min_{x \in \R^n}~\Ex_{\xi\sim \mm_{x}}(f(x,\xi)) + g(x),
\end{equation} 
In this case, one tries to learn a decision rule from a decision-dependent data distribution $\mm_x$. Problems of the form \eqref{eq:optim_intro_x} were addressed in the framework of performative prediction proposed in \cite{Perdomo&al,Mendler&al} and discussed with further algorithmic aspects in \cite{Drusvyatskiy&Xiao}. A typical example concerns prediction of loan default risks, that is the chance that a borrower won't be able to repay their loan. More precisely, banks take into account several parameters, including the default risk, to decide whether to accept a consumer's loan application and, if so, what interest rate will apply. It is clear that a high default risk implies a high interest rate, but a high interest rate increases the consumer's default risk. Thus, the predictive performance of the bank's model is not calibrated with respect to future results obtained by acting on the model. Another example concerns navigation apps, such as Google Maps (see, \eg \cite{ETA-Prediction, macfarlane2019your}), which  suggest routes with low travel time to users. This influences users' decisions to pick such routes and consequently, increases traffic on these routes, impacting travel times. Further applications and illustrations can be found in \cite[Appendix A]{Perdomo&al}. 

In general, problems of the form \eqref{eq:optim_intro_x} are difficult to solve. However, a natural approach consists in performing a repeated minimization procedure, \ie throughout iterations, one solves 
	\begin{equation}
		\label{eq:rrm_intro}
		x_{t+1} \in \argmin_{x \in \R^n} \Ex_{\xi\sim\mm_{x_t}}(f(x,\xi)) + g(x),
	\end{equation}
and then updates the distribution $\mm_{x_{t+1}}$. Under suitable assumptions that will be specified later, the sequence $(x_t)_t$ generated by the repeated minimization procedure \eqref{eq:rrm_intro} admits a fixed point $\bar{x}$. Such a point turns out to be an equilibrium with respect to the distribution $\mm_{(.)}$ in the following sense:
\begin{equation}
	\label{eq:equilibrium_intro}
	\bar{x} \in \argmin_{x \in \R^n}~\Ex_{\xi\sim \mm_{\bar{x}}}(f(x,\xi)) + g(x),
\end{equation}
 that is, $\bar{x}$ solves \eqref{eq:optim_intro_x} for the induced distribution $\mm_{\bar{x}}$. So instead of solving \eqref{eq:optim_intro_x}, we look at an equilibrium point in the sense of \eqref{eq:equilibrium_intro}. Notice that in terms of operators, \eqref{eq:equilibrium_intro} can be written (formally, for instance) as the monotone inclusion:
 \begin{equation}\label{eq:inclusion_intro}
 A(\bar x)+ B(\bar{x})\ni 0
 \end{equation}
 with $A(x) = \partial g(x)$ and $B(x) = \Ex_{\xi\sim\mm_{\bar x}}(\nabla f(x,\xi))$ ($\nabla$ being always the gradient w.r.t. to the variable $x$), where $A+B$ is monotone. That is $\bar{x}$ is a zero of the sum of two monotone operators. One strategy to solve problems of the form \eqref{eq:inclusion_intro} is to consider some continuous and discrete dynamical systems whose trajectories may converge, under suitable assumptions, to an element in $(A+B)^{-1}(0)$, the zero set of the sum $A+B$. For instance, when $B\equiv 0$ and $A = \partial g$ where $g$ is a proper convex lower semicontinuous function on $\R^n$, it is well known since the works of Brézis, Baillon and Bucker \cite{Baillon-Brezis,Bruck}, that each trajectory of the subgradient flow
 \begin{equation}
 	\dot{x}(t) + \partial g(x) \ni 0 ~~\mbox{with}~~x(0) = x_0\in\R^n,
 \end{equation}
 converges to a minimizer of $g$, and thus a zero of $\partial g$, provided $\argmin g\neq\emptyset$.
 
 Designing algorithms and dynamical systems with rapid convergence properties to solve monotone inclusions is at the core of many fields in modern optimization, partial differential equation, game theory, etc. The literature is extensive, and to name only a few, the reader is referred to \cite{Alvarez&Attouch,Attouch&Peypouquet,Su&al,Alvarez,Attouch&Cabot&Redont,Attouch-Chbani-Fadili-Riahi,Bot&al, kim2021accelerated,Lin&Jordan} and the references therein.\\

 In this work, we address closed-loop differential inclusions of the form 
 \begin{equation}\label{eq:MI_intro}
 \dot{x}(t) + A(x(t)) + {\Ex_{\xi\sim \mm_{x(t)}}}\left(B(x(t),\xi)\right)\ni 0,
 \end{equation}
 where $A$ is a maximally monotone operator, and $B(\cdot,\xi)$ is a single-valued mapping for all $\xi \in \Xi$. The particularity of such a dynamic is, of course, the presence of  the random operator ${\Ex_{\xi\sim \mm_{(.)}}}\left(B(.,\xi)\right)$ where the random variable $\xi$ has a trajectory-dependent distribution $\mm_{x(t)}$. Thus, it is not straightforward how to  address \eqref{eq:MI_intro} within the classical framework (see, \eg \cite{brezismaxmon,Barbu}). Yet, a clever reformulation of \eqref{eq:MI_intro}, based on the notion of equilibrium,  as a monotone inclusion governed by a Lipschitz perturbation of a maximally monotone operator will allow us to tackle this issue. Then we investigate inertial dynamics related to \eqref{eq:MI_intro}. Indeed, since the work of Polyak \cite{polyak1964some}, who considered a system of the form
 \begin{equation}\label{eq:polyak_intro}
 \tag{HBF}
 			\ddot{x}(t) + \gamma\dot{x}(t)  + \nabla f(x(t)) = 0,
 \end{equation}
 where $\gamma>0$ is called the viscous damping coefficient, the introduction of inertial dynamics to accelerate  optimization methods has gained a lot of attention and led to many developments (see, \eg \cite{attouch2000heavy,alvarez2002second,attouch2016fast,attouch2014dynamical,attouch2018fast,Su&al} and the references therein).
 
In this paper, we then consider second-order dynamics of the form
 	\begin{equation}\label{eq:MI_intro_2nd_order}
		\ddot{x}(t) + \gamma(t)\dot{x}(t)  + \nabla f_{\mm_{\bar{x}}}(x(t)) + \omega \nabla^2  f_{\mm_{\bar{x}}}(x(t)) \dot{x}(t) +  \e_{\bar{x}}(x(t)) + \omega\frac{\dd }{\dd t}\e_{\bar{x}}(x(t)) = 0,
	\end{equation}
	where $f_{\mm_{\bar{x}}}(x) = \Ex_{\xi\sim\mm_{\bar{x}}}(f(x,\xi))$, $\e_{\bar{x}}: \cH \to \cH$ is a perturbation operator and $\omega$ is the so-called Hessian-driven damping coefficient.
When $\gamma(t)\equiv\gamma$, $f_{\mm_{\bar{x}}} = f$ (\ie without a stochastic structure) and $\e_{\bar{x}}(x) = 0$, systems of the form \eqref{eq:MI_intro_2nd_order} were first studied in \cite{alvarez2002second}. Later, this system was combined with an asymptotic vanishing damping $\gamma(t) = \frac{\alpha}{t}$, for $\alpha>0$ in \cite{attouch2016fast}. Several recent studies have been devoted to this topic (see, \eg \cite{kim2021accelerated,Shi&al,Bot&al,Lin&Jordan,Attouch-Chbani-Fadili-Riahi}).

\subsection{Statement of the problem}
 Throughout, $\cH$ is a real Hilbert space endowed with the scalar product $\langle .,.\rangle$ and induced norm $\Vert\cdot\Vert$, and $\Xi$ is a Polish space, \ie separable and completely metrizable. 

We consider the closed-loop differential inclusion
\begin{equation}
	\label{eq:eq0}
	\left\lbrace
	\begin{aligned}
		&\dot{x}(t) + A(x(t)) + {\Ex_{\xi\sim \mm_{x(t)}}}\left(B(x(t),\xi)\right)\ni 0,~\text{a.e}~t > t_0>0, \\
		&x(t_0) = x_0 \in \overline{\dom(A)},
	\end{aligned}
	\right.
\end{equation}
where $\mm_{x}$ is a family of probability distributions on $\Xi$ indexed by $x\in \cH$. We will work under the standing assumption:
\begin{assumption}\label{assump:1}{~}\medskip
\begin{itemize}
\item $A: \cH \rightrightarrows \cH$ is a set-valued maximal monotone operator such that $\inter(\dom(A))\neq \emptyset$;
\item $B: \cH \times \Xi \to \cH$ is single-valued with $\xi\mapsto B(x,\xi)$ $\mm_x$-measurable, and $\exists \beta > 0$ such that $\xi\mapsto B(x,\xi)$ is $\beta$-Lipschitz continuous for every $x\in \cH$; 
\end{itemize}
\end{assumption}
Observe that when $x \mapsto \mm_{x}(C)$ is a measurable function on $\cH$ for each fixed $C \in \cB$, where $\cB$ is a countably generated $\sigma$-field on $\Xi$, $(\mm_{x})_{x \in \cH}$ can be viewed as a random walk on $(\cH \times \Xi,\mathcal{F} \otimes \cB)$ where $\mathcal{F}$ is a $\sigma$-field on $\cH$; see Section~\ref{section:ricci}.
\medskip

\begin{example}
	\begin{enumerate}[1.]
		\item  Typically, \eqref{eq:eq0} covers the case of stochastic optimization problems with a state-dependent distribution studied in \cite{Perdomo&al,Drusvyatskiy&Xiao} by taking $A = \partial g$ and $B = \nabla f$ for $g\in\Gamma_{0}(\cH)$ and $f \in C^1(\cH \times \Xi)$ whose gradient $\nabla f(x,\cdot)$ is $\beta$-Lipschitz continuous for every $x \in \cH$. The last assumption ensures that ${\Ex_{\xi\sim \mm}}\left(f(\cdot,\xi)\right)$ is $C^1(\cH)$ whose gradient is ${\Ex_{\xi\sim \mm}}\left(\nabla f(\cdot,\xi)\right)$ which is $\beta$-Lipschitz continuous.
		\item Taking $A = N_{K}$ the normal cone of a nonempty closed convex set $K$ of admissible decisions, we recover the framework of variational inequalities addressed recently in \cite{cutler2022stochastic}.
	\end{enumerate}
\end{example}

To simplify the presentation, we set, for any measure $\mm\in\Pro$
\begin{equation}\label{eq:f_mu}
	B_{\mm}(x) =  {\Ex}_{\xi\sim \mm}\left(B(x,\xi)\right)~\mbox{and}~F_{\mm}(x)  =  A(x) + B_{\mm}(x) .
\end{equation}

Using this notation, the system \eqref{eq:eq0} can be simply rewritten as
\begin{equation}
	\tag{SMI}\label{eq:eq1}
	\left\lbrace
	\begin{aligned}
		\dot{x}(t) &+ F_{\mm_{x(t)}}(x(t))\ni 0,~\text{a.e}~t > t_0,\\
		x(t_0) &= x_0 \in \overline{\dom(A)}.
	\end{aligned}
	\right.
\end{equation}
 The acronym \eqref{eq:eq1} for the above dynamic stands for Stochastic Monotone Inclusion. Though "monotone" may seem as an abuse of terminology because the measure $\mm_{(.)}$ depends on the trajectory, and $B_{\mm}$ is not even monotone. We will show later that this terminology is still justified as \eqref{eq:eq1} can be reformulated as a Lipschitzian perturbation of a monotone inclusion (see Section~\ref{subsec:wellposed}).

\subsection{ Contributions and organization of the paper}
The paper is organized as follows. In \cref{section:firstorder} we address first-order monotone inclusions with closed-loop distributions. We prove the existence of equilibria as well as the well-posedness of the dynamics and  convergence properties of the trajectories. In \cref{section:hessian2ndorder} we study asymptotic convergence properties of the trajectories of second-order dynamics with closed-loop distributions via viscous and Hessian damping. This allows us in particular to cover problems of the form \eqref{eq:optim_intro_x}. \cref{section:ricci} contains a discussion concerning the Lipschitz behavior of the family $(\mm_x)_{x \in \cH}$ with respect to the Wasserstein distance and some consequences in the framework of Markov chains on metric random walk spaces. In \cref{section:application} we discuss the inertial primal-dual algorithm as an application of our results. Finally, \cref{section:conclusion} contains some conclusions and discusses some future works.

\section{Notation and preliminaries}
In this section we fix some notation and present some notions and results that will be used. 
\subsection{Convex analysis}
The domain of a function $g$ on $\cH$ is defined by $\dom(g) = \{x\in\cH:~g(x)<\infty\}$. We denote by $\Gamma_{0}(\cH)$ the class of proper (bounded from below and $\dom(g) \neq \emptyset$), lower semicontinuous (l.s.c) and convex functions on $\cH$ with values in $\R\cup\{+\infty\}$. We say that $g$ is $\alpha$-strongly convex, for $\alpha>0$, if $g - \frac{\alpha}{2}\Vert .\Vert^{2}$ is convex. 

The subdifferential of $g$ is defined as
\[
\partial g:x\in\cH\mapsto\{ v\in\cH:~g(y)\geq g(x) + \langle v,y-x\rangle \}.
\]
We recall the following Fermat's optimality condition for $g \in \Gamma_0(\cH)$,
\[
0\in\partial g(x^{\star}) \Leftrightarrow x^\star \in \argmin g(\cH).
\]
\begin{definition}[Differentiability]
	Let $g:\cH\to\R\cup\{\infty\}$ and $x\in\textup{int}(\dom(g))$. We say that $g$ is (Fr\'echet) differentiable at $x$ if there exists $v\in\cH$ such that
	\[
	\lim_{h\to 0} \frac{g(x+h) - g(x) - \langle v,h\rangle}{\Vert h\Vert} = 0.
	\]
	The unique vector $v$ satisfying this condition is the gradient of $g$ at $x$ denoted by $\nabla g(x)$.
\end{definition}
If $g \in \Gamma_{0}(\cH)$ and differentiable at $x$, then $\partial g(x)=\{\nabla g(x)\}$.\\

 We also recall the following.
\begin{definition}[$L$-smoothness]
	Let $L\geq 0$ and  $g:\cH\to\R\cup\{\infty\}$. We say that $g$ is $L$-smooth over $D\subset\cH$ if it is differentiable over $D$ and 
	\[
	\Vert \nabla g(x) - \nabla g(y)\Vert \leq L \Vert x - y\Vert~\mbox{for any}~x,y\in D.
	\]
	We denote by $C^{1,1}_{L}(D)$ the class of $L$-smooth functions over $D$.
\end{definition}
Given a function $g:\cH\to\R\cup\{+\infty\}$, it proximal mapping is defined through
\[
\prox_{f}(x) = \argmin_{y \in \cH}\left\{g(y) + \frac{1}{2}\Vert y - x\Vert^{2}\right\}~\mbox{for any}~x\in\cH.
\]
When $g = \I_{K}$ the indicator function of a nonempty closed convex set $K\subset\cH$, then $\prox_{f} = \proj_{K}$, the projector onto $K$.
For further details and notion, we refer the reader to \cite{Baucshke&Combettes}.

\subsection{Operator theory}
The domain of the set-valued operator $A:\cH\rightrightarrows\cH$ is $\dom(A) = \{x\in\cH:~A(x)\neq \emptyset\}$, its graph is $\gra (A) = \{ [x,u]\in\cH\times\cH:~u\in Ax\}$ and its zeros set is $\zer (A)  = \{x\in\cH:~0\in A(x)\}:= A^{-1}(0)$. 
A selection of $A$ is an operator $T:\dom A\to\cH$ such that, $Tx\in Ax$ for any $x\in\dom A$. We write $A:\cH\to\cH$ to indicate that $A$ is single-valued. In the following, we gather some main properties that are essential for the rest of the paper.
\begin{definition}\label{def:op_properties}
	\begin{itemize}
		\item We say that $A$ is $\beta$-Lipschitz continuous if it is single-valued over its domain and
		\begin{equation}\label{eq:Lip_op}
			\Vert Ax - Ay \Vert \leq \beta \Vert x-y\Vert , \quad \forall x,y\in \dom A.
		\end{equation}
		\item We say that $A:\cH\rightrightarrows\cH$ is monotone if
		\begin{equation}\label{eq:monotone_op}
			\langle x-y,u-v\rangle \geq 0 , \quad \forall [x,u],[y,v]\in\gra A.
		\end{equation}
		\item We say that $A$ is maximal monotone if there exists no monotone operator $B$, \ie satisfying \eqref{eq:monotone_op}, such that $\gra A\subset \gra B$.
		\item We say that $A$ is uniformly monotone with modulus $\phi:[0,\infty)\to [0,\infty)$ if $\phi$ is increasing, $\phi(0) = 0$, $\lim_{t\to\infty}\phi(t) = \infty$ and
		\begin{equation}\label{eq:uniformly_monotone_op}
			\langle x-y,u-v\rangle\geq \Vert x-y\Vert \phi\left(\Vert x-y\Vert\right) , \quad \forall [x,u],[y,v]\in\gra A.
		\end{equation}
		
		\item  We say that $A$ is $\mu$-strongly monotone, with $\mu >0$, if 
		\begin{equation}\label{eq:str_monotone_op}
			\langle x-y,u-v\rangle\geq \mu \Vert x-y\Vert^{2} , \quad \forall [x,u],[y,v]\in\gra A.
		\end{equation}
			\end{itemize}
\end{definition}
\begin{remark}
	\begin{itemize}
		\item  Note that if $A$ is $\mu$-strongly monotone is equivalent to saying that $A-\mu\Id$ is monotone.
				\item The definition of uniform monotonicity given by \eqref{eq:uniformly_monotone_op} is slightly different from the one in \cite[Definition 22.1]{Baucshke&Combettes}. 
		\item If $A$ is $\mu$-strongly monotone, then it is uniformly monotone with modulus $\phi(t) = \mu t$.
	\end{itemize}
\end{remark}
\begin{example}
	\begin{itemize}
		\item The typical example of a maximal monotone operator is the subdifferential $\partial g$ of a function $g\in\Gamma_{0}(\cH)$. We usually refer to such an operator as a subpotential maximal monotone operator.
	\end{itemize}
\end{example}

\subsection{Monotone inclusions}
Let $A$ be a maximal monotone operator on $\cH$ and a single-valued mapping $D:[t_0,+\infty[\times\overline{\dom(A)} \rightarrow \cH$ and consider the following differential inclusion
\begin{equation}\label{eq:solution_MI}
\left\{\begin{array}{l}
\dot{x}(t)+A(x(t))+D(t,x(t)) \ni 0, \quad t \in [t_0, T] \\
x(t_0)=x_0\in\dom(A).
\end{array}\right.
\end{equation}

\begin{definition}\label{def:strong-solution}
We will say that $x:[t_0,T]\to \cH$ is a strong solution trajectory on $[t_0,T]$ of \eqref{eq:solution_MI} if the following properties are satisfied:
\begin{enumerate}[label=(\alph*)]
\item \label{defsolforte-1}
$x$ is continuous on $[t_0,T]$ and absolutely continuous on any compact subset of $]t_0,T[$ (hence almost everywhere differentiable);
\item \label{defsolforte-3}
$x(t)\in\dom(A)$ for almost every $t\in]t_0,T]$, and \eqref{eq:solution_MI} is verified for almost every $t\in]t_0,T[$.
\end{enumerate}

A trajectory $x:[t_0,+\infty[\to\cH$ is a strong global solution of \eqref{eq:solution_MI} if it is a strong solution on $[t_0,T]$ for any $T>t_0$.
\end{definition}

For further details, we refer the reader to the classical monographs \cite{brezismaxmon} or \cite{Barbu}.

\subsection{Transportation distance}
Denote $\Pro$ the space of probability measures on $\Xi$. For $\mm_1,\mm_2\in\Pro$, the $\Wass_{1}$-Wasserstein distance is defined by
\begin{equation}\label{eq:Wasserstein1}
	\Wass_{1}(\mm_1,\mm_2) = \sup_{h\in\Lip_1} \left|\Ex_{\xi\sim\mm_1}h(\xi) - \Ex_{\zeta\sim\mm_2}h(\zeta)\right|,
\end{equation}
where $\Lip_1$ is the space of $1$-Lipschitz continuous functions $h: \Xi \to \R$.
\section{First-order monotone inclusions}\label{section:firstorder}

In this section we perform the analysis of the first-order monotone inclusion \eqref{eq:eq0}. More precisely, we discuss the existence and uniqueness of solutions as well as the convergence of trajectories. Recall that the dynamic \eqref{eq:eq0} is governed by the operator $F_{\mm_x} = A + B_{\mm_x}$. We first prove the existence of an equilibrium point $\bar{x}$ which will allow us to reformulate \eqref{eq:eq0} in a suitable form.

The following assumption is essential for the convergence analysis. It describes the sensitivity of the distribution to shifts in the index (here trajectory). It is widely used in the literature (see, \eg \cite{Perdomo&al,Mendler&al, Drusvyatskiy&Xiao,wood2023stochastic}). We will discuss how it closely relates to the so-called coarse Ricci curvature in \cref{section:ricci}.
\begin{assumption}[Lipschitz distributions]\label{assump:2}
	There exists $\tau >0$ such that
	\[
	\Wass_{1}(\mm_x,\mm_y) \leq \tau \Vert x - y\Vert,~\mbox{for all}~x,y\in \cH.
	\]
\end{assumption}

The following two assumptions are standard monotonicity assumptions and will be crucial for the well-posedness of the dynamics and to prove existence and uniqueness of equilibria.
\begin{assumption}[Strong monotonicity]\label{assump:3}
	$\exists \mu > 0$ such that for all $x\in\cH$, $F_{\mm_{x}}$ is $\mu$-strongly monotone.
\end{assumption}
The strong monotonicity \cref{assump:3} can be weakened to uniform monotonicity in the following sense.
\begin{assumption}[Uniform monotonicity]\label{assump:4}
	There exists a function $\phi$ satisfying
	\[
	\phi(t) > \beta\tau t,~\forall t>0 ,
	\]
	such that for all $x\in\cH$, $F_{\mm_{x}}$ is uniformly monotone with a modulus $\phi$.
\end{assumption}
Finally, define the following parameter $\rho:=\frac{\beta\tau}{\mu}$. As we will see, (see also \cite{Drusvyatskiy&Xiao,Perdomo&al}), the parameter regime $\rho <1$ will play a crucial role in the analysis of the convergence of the trajectories.

\subsection{Existence and uniqueness of equilibria}
	\begin{definition}(Equilibrium point) 
		We say that $\bar{x}\in\cH$ is at equilibrium with respect to the family of probability measures $(\mm_x)_{x \in \cH}$ if
		\begin{equation}\label{eq:equilibrium}
			0\in \  F_{\mm_{\bar{x}}}(\bar{x}) .
		\end{equation}
	\end{definition}
	
	In case $A = \partial g$ and $B = \nabla f$ where $g\in\Gamma_{0}(\cH)$ and $f(x,\cdot)\in C^{1,1}_{\beta}(\Xi)$, this definition is to be compared to the one introduced in  \cite{Perdomo&al} (see also \cite{Drusvyatskiy&Xiao}). Indeed, \eqref{eq:equilibrium} reduces to:
	\begin{equation}\label{eq:equil_example}
		\bar{x}\in\argmin_{x \in \cH} \Ex_{\xi\sim\mm_{\bar{x}}}(f(x,\xi)) + g(x).
	\end{equation}
	Solutions of  \eqref{eq:equil_example} are exactly the fixed points of the repeated minimization procedure, that is, starting from some $x_0$, we generate the following sequence for $t\geq 0$
	\begin{equation}
		\label{eq:rrm}
		x_{t+1} = S(x_t) :=\argmin_{x \in \cH} \Ex_{\xi\sim\mm_{x_t}}(f(x,\xi)) + g(x).
	\end{equation}

	In \cite[Theorem 3.5]{Perdomo&al} it is shown that if $f$ is $C^1$ in both variables,  $\xi\mapsto \nabla f(x,\xi)$ is $\beta-$Lipschitz and $\Ex_{\xi\sim\mm_{x}}(f(.,\xi))$ is $\mu$-strongly convex for all $x\in \cH$ with $\rho<1$, then, under \cref{assump:2}, the iterates of \eqref{eq:rrm} converge to a unique stable point. Their proof is essentially based on a fixed point argument. In \cite[Proposition 4.1]{Perdomo&al}, they show the existence of equilibrium points under weaker assumptions on the loss $f$.  Specifically, they demonstrate that if $f$ is convex and jointly continuous, then equilibrium points exist provided $\dom(g)$ is compact. However, in this case the equilibrium is not necessarily unique. In the following lemma, we show the existence of equilibrium in the sense of \eqref{eq:equilibrium}.

	\begin{theorem}[Existence and uniqueness of equilibrium point]\label{thm:equilibrium}
		Under \cref{assump:4}, the map
		\[
		S:x\in\cH\mapsto \zer(F_{\mm_{x}})=\{u\in\cH:~0\in F_{\mm_{x}}(u)\},
		\]
		is a contraction. In particular, the equilibrium $\bar{x}$ is unique. If moreover, \cref{assump:3} holds instead, \ie  $F_{\mm_{x}}$ is $\mu$-strongly monotone for $\mu>0$, the mapping $S$
		is $\rho$-Lipschitz with $\rho:=\frac{\beta\tau}{\mu}$. Thus for $\rho < 1$, there exists a unique equilibrium point $\bar{x}$.
	\end{theorem}
	\begin{proof}
		First, we see that $S$ is well defined. Indeed, for any $x\in\cH$, $\zer (F_{\mm_{x}})$ is nonempty, and is in fact a singleton due to \cref{assump:4}. To see this, we argue as in \cite[Proposition 22.11]{Baucshke&Combettes}. Fix $[y_0,u_0]\in \gra F_{\mm_{x}}$. We have for any $[y,u]\in \gra F_{\mm_{x}}$:
		\begin{equation}
			\begin{aligned}
				\Vert y-y_0\Vert  \Vert u\Vert \geq \langle y-y_0,u\rangle &=    \langle y-y_0,u - u_0\rangle +  \langle y-y_0,u_0\rangle\\
				&\geq \Vert y-y_0\Vert \phi\left(\Vert y-y_0\Vert \right) - \Vert y-y_0\Vert \Vert u_0\Vert.
			\end{aligned}
		\end{equation}
		Since $\lim_{t\to\infty}\phi(t) = \infty$ we infer that $\inf_{u\in \gra F_{\mm_{x}}(y)}\Vert u\Vert\to\infty$ as $\Vert y\Vert\to\infty$ and thus $F_{\mm_{x}}$ is surjective (see \cite[Corollary 21.25]{Baucshke&Combettes}). Moreover, in view of strict monotonicity of $F_{\mm_{x}}$,  $\zer (F_{\mm_{x}})$ is a singleton (see, \eg \cite[Proposition 23.35]{Baucshke&Combettes}).
		
		Now, pick $x,y\in\cH$. We have that $0\in F_{\mm_{x}}(S(x))$ and $0\in F_{\mm_{y}}(S(y))$. In particular, $-B_{\mm_{y}}(S(y)) \in A(S(y))$, which gives that $B_{\mm_{x}}(S(y) ) - B_{\mm_{y}}(S(y))\in F_{\mm_{x}}(S(y))$. We get, thanks to \cref{assump:4} 
		\[
		\Vert S(x) - S(y)\Vert\phi\Big(\Vert S(x) - S(y)\Vert\Big) \leq \langle u - v,S(x) - S(y)\rangle,~\mbox{for any}~(u,v)\in F_{\mm_{x}}(S(y))\times F_{\mm_{x}}(S(x)).
		\]
		Then, taking $u = B_{\mm_{x}}(S(y) ) - B_{\mm_{y}}(S(y))$ and $v = 0$, we get, using Cauchy-Schwarz inequality
		\[
		\phi\Big(\Vert S(x) - S(y)\Vert\Big) \leq \Vert B_{\mm_{x}}(S(y)) - B_{\mm_{y}}(S(y))\Vert.
		\]
		We get, using \cref{corollary:1} below
		\begin{equation}\label{eq:inequality_psi}
			\phi\Big(\Vert S(x) - S(y)\Vert\Big) \leq \beta\tau\Vert x- y\Vert.
		\end{equation}
		Since $\phi$ is strictly increasing, the last inequality gives, thanks to \cref{assump:4},
		\[
		\Vert S(x) - S(y)\Vert \leq \phi^{-1}\Big(\beta\tau\Vert x- y\Vert\Big) < \Vert x- y\Vert,
		\]
		and by the Banach fixed-point theorem (see \cref{thm:fixedpoint}), $S$ has a unique fixed point $\bar{x}$.
		
		Now if $F_{\mm_{x}}$ is $\mu$-strongly monotone, it is in particular uniformly monotone with modulus $\phi(t) = \mu t$. Equation \eqref{eq:inequality_psi} gives
		\[
		\Vert S(x) - S(y)\Vert \leq \rho\Vert x- y\Vert,
		\]
		with $\frac{\beta\tau}{\mu} :=\rho$. Again, we conclude using \cref{thm:fixedpoint}.
	\end{proof}
	\begin{remark}
		In the strongly monotone case, \cref{assump:4} incorporates the parameter regime $\rho < 1$ which appears in particular in \cite[Theorem 3.5]{Perdomo&al}.
	\end{remark}

\subsection{Well-posedness}\label{subsec:wellposed}
Before stating the main result of this section, let us fix some extra notation and properties. Assume that \cref{assump:4} holds, and denote by 
	\begin{equation}\label{eq:phi}
			\varphi(t) = \phi(t) - \beta \tau t,
				\end{equation}
		where $\phi$ is the modulus of uniform monotonicity of $F_{\mm_{\bar{x}}}$.

	We prove the following.
	\begin{lemma}\label{lemma:M}
		Let $a>0$ and define 
		\begin{equation}\label{eq:modulus}
			\theta(z) := \int_{z}^{a}\frac{\dd s}{\varphi(s)}.
		\end{equation}
		Then $\theta$ is nonincreasing and $\lim_{z\to 0^{+}}\theta(z) = \infty$.
		
	\end{lemma}
	\begin{proof}
		Indeed, we have $\dot{\theta}(z) = -1/\varphi(z) <0 $ since $\phi$ satisfies \cref{assump:4}. Moreover, since $\phi(t) \geq \varphi(t)$, and  $\lim_{z\to 0^{+}}\int_{z}^{a}\frac{\dd s}{\phi(s)} = \infty$, the result follows.	
	\end{proof}
	\begin{remark}
		In the literature of ordinary differential equations, the above lemma is related to the fact that $\varphi$ is somehow an \textit{Osgood modulus of continuity} (see, \eg \cite[Definitions 2.108 and 3.1]{Bahouri&al}). If $\varphi(s) = s$, which corresponds  to Lipschitz regularity, and $a=1$ then $\theta(z) = \log_{+}(z) = \max\{0,\log(1/z)\}$. If $a = 1/e$ and $\varphi(s) = s\log(1/s)$, which corresponds  to log-Lipschitz regularity,  then $\theta(z) = \log\log_{+}(z)$. More generally, $\varphi(s) = s\left(\log(1/s)\right)^{r}$ for $r\leq 1$ are admissible choices.
	\end{remark}

Now let us define the following gap
\begin{equation}\label{eq:e_x}
	\e_{\bar{x}}(x) =  B_{\mm_x}(x) - B_{\mm_{\bar{x}}}(x).
\end{equation}
Using the notation in \eqref{eq:f_mu}, we may rewrite \eqref{eq:eq0} in the following form
\begin{equation}\tag{p-SMI}\label{eq:eq2}
	\left\lbrace
	\begin{aligned}
		\dot{x}(t) &+ F_{\mm_{\bar{x}}}(x(t)) + \e_{\bar{x}}(x(t))\ni 0,~\text{a.e}~t > t_0\\
		x(t_0) &= x_0.
	\end{aligned}
	\right.
\end{equation}
One advantage of this formulation is that the mapping $x\mapsto \e_{\bar{x}}(x)$ exhibits Lipschitz behaviour (see \cref{lemma:3}), and now only the operator $F_{\mm_{\bar{x}}}$ appears instead of $F_{\mm_{x(t)}}$. This allows us to treat \eqref{eq:eq2} within the framework of evolution equations  governed by Lipschitz perturbations of maximal monotone operators (\cf \cite[Chapter III]{brezismaxmon}). This is behind the notation \eqref{eq:eq2}, which stands for perturbed stochastic monotone inclusion.  This being said, our aim is to  use \cite[Proposition 3.13]{brezismaxmon} and show the existence of a unique strong solution (\cf \cref{def:strong-solution}) to \eqref{eq:eq2} (see also \cite[Definition 3.1]{brezismaxmon}). 

 We start with the following properties.
\begin{lemma}\label{lemma:1}
	Under \cref{assump:1}, we have, for any $\mm,\nu\in\Pro$ and $x \in \cH$
	\[
	\sup_{x \in \cH}\Vert B_{\mm}(x)  - B_{\nu}(x)\Vert \leq \beta\Wass_1(\mm,\nu).
	\]
\end{lemma}
\begin{proof}
	Let $v \in \cH$ be a unit norm vector. By \cref{assump:1}, $\langle v,B(x,\cdot) \rangle$ is $\beta$-Lipschitz continuous for every $x \in \cH$. We then have from the definition of the $\Wass_{1}$-Wasserstein distance \eqref{eq:Wasserstein1} 
	\begin{align*}
	\|B_{\mm}(x) - B_{\nu}(x)\| 
	&= \sup_{v \in \cH,\|v\|=1}\langle v,B_{\mm}(x) - B_{\nu}(x)\rangle \\
	&= \sup_{v \in \cH,\|v\|=1}\Ex_{\xi\sim \mm} \langle v,B(x,\xi)\rangle - \Ex_{\zeta\sim \nu} \langle v,B(x,\zeta)\rangle \leq \beta \Wass_1(\mm,\nu).
	\end{align*}
	
	Taking the supremum over $x\in\cH$ yields the result.
\end{proof}
Combining \cref{lemma:1} and \cref{assump:2} we get the following.
\begin{corollary}\label{corollary:1}
	Under \cref{assump:1} and \cref{assump:2}, for all $y,z\in \cH$
	\[
	\sup_{x \in \cH}\Vert B_{\mm_{y}}(x)  - B_{\mm_{z}}(x)\Vert \leq \beta\tau\Vert y-z\Vert.
	\]
\end{corollary}

\begin{assumption}[Lipschitz continuity of $B_{\mm_{x}}$]\label{assump:5}
$\exists L > 0$ such that $B_{\mm_x}$ is $L$-Lipschitz continuous for every $x \in \cH$. 
\end{assumption}
This assumption is true if for instance $B(\cdot,\xi)$ is $L$-Lipschitz continuous for $\mm_x$-almost every $\xi \in \Xi$. Indeed
\[
B_{\mm_x}(z) - B_{\mm_x}(y) = \Ex_{\xi\sim \mm_x}\left(B(z,\xi) - B(y,\xi)\right) \leq L \Vert z-y\Vert.
\]

We are now in position to characterize the properties of $F_{\mm_{\bar{x}}}$.
\begin{lemma}\label{lemma:0}
Assume that \cref{assump:1,assump:5} hold and that $B_{\mm_{\bar{x}}}$ is monotone. Then the operator $F_{\mm_{\bar{x}}}$ is maximally monotone. If, in addition, \cref{assump:3} (resp. \cref{assump:4}) then $F_{\mm_{\bar{x}}}$ is maximally uniformly (resp. strongly) monotone.
\end{lemma}
\begin{proof}
Maximal monotonicity of $B_{\mm_{\bar{x}}}$ follows from \cite[Corollary~20.28]{Baucshke&Combettes}. Combine this with maximal monotonicity of $A$ stated in \cref{assump:1}, and \cite[Lemma~2.4]{brezismaxmon} yields the claim. The proof of the last statement is immediate.
\end{proof}

The key ingredient to use \cite[Proposition~3.13]{brezismaxmon} is the Lipschitz continuity of the perturbation $\e_{\bar{x}}(.)$. This is the content of the following statement.
\begin{lemma}\label{lemma:3}
	Suppose that \cref{assump:1,assump:2,assump:5} hold, then $\e_{\bar{x}}(.)$ is $(2L+\beta\tau)$-Lipschitz continuous.
\end{lemma}
\begin{proof}
	Let $x,z\in\cH$. We then have
	\begin{equation}
		\begin{aligned}
			\Vert\e_{\bar{x}}(x) - \e_{\bar{x}}(z)\Vert &= \Vert\left(B_{\mm_{x}}(x) - B_{\mm_{\bar{x}}}(x)\right) -  \left(B_{\mm_{z}}(z) - B_{\mm_{\bar{x}}}(z)\right)\Vert \\
			&= \Vert\left(B_{\mm_{x}}(x) - B_{\mm_{x}}(z) \right) + \left(B_{\mm_{\bar{x}}}(z)- B_{\mm_{\bar{x}}}(x)\right) + \left(B_{\mm_{x}}(z) - B_{\mm_{z}}(z)\right)\Vert\\
					&\leq L \Vert x-z\Vert + L \Vert x-z\Vert + \beta \tau\Vert x-z\Vert\\
			&=   (2L+\beta\tau)\Vert x-z\Vert,
		\end{aligned}
	\end{equation}
	where we have used \cref{assump:5} twice and \cref{corollary:1} once in the inequality.
\end{proof}
	\begin{remark}
		Since $\zer (F_{\mm_{\bar{x}}})  = \{\bar{x}\}$, we clearly see that  $\bar{x}\in \zer (F_{\mm_{\bar{x}}} + \e_{\bar{x}})$. Indeed, taking into account \eqref{eq:e_x}, we have $\e_{\bar{x}}(\bar{x}) = 0$ so that $0\in \left(F_{\mm_{\bar{x}}}+\e_{\bar{x}}\right)(\bar{x})$.
	\end{remark}
	\begin{proposition}\label{prop:1st_order_wp}
		Suppose that \cref{assump:1,assump:2,assump:5} hold, that either \cref{assump:3} or \cref{assump:4} hold and let $T > t_0$. Then, given $x_0\in \overline{\dom(F_{\mm_{\bar{x}}})}$, the dynamic \eqref{eq:eq2}, and hence \eqref{eq:eq1}, admits a unique strong solution $x$ in the sense of \cref{def:strong-solution}.
	\end{proposition}
	\begin{proof}
		Thanks to \cref{lemma:0} and \cref{lemma:3} $F_{\mm_{\bar{x}}}$ is maximally monotone and $\e_{\bar{x}}$ is Lipschitz continuous. Since $t\mapsto \e_{\bar{x}}(x)$ is trivially  in $L^{\infty}(t_0,T;\cH)$ for any $x \in \cH$, we deduce, thanks to \cite[Proposition~3.13]{brezismaxmon} the existence of a unique solution $x\in W^{1,1}(t_0,T;\cH)$ to \eqref{eq:eq1}.
	\end{proof}
	
\begin{remark}
Evolution problems of the form \eqref{eq:eq2} were also studied in \cite{Attouch&Damlamian}, where the operator $\e_{\bar{x}}$ is potentially multivalued and depends on both time and space variables. While standard time-variable measurability is assumed, the authors' analysis relied on the upper semicontinuity of $x\mapsto \e_{\bar{x}}(t,x)$, rather than Lipschitz continuity, for which weaker results were proved. We also observe that the domain condition in \cref{assump:1} can be removed in finite dimension (see \eg \cite[Theorem~1.2]{Attouch&Damlamian}) or when $A$ is the subdifferential of a function in $\Gamma_0(\cH)$.
\end{remark}

\subsection{Convergence properties}\label{subsec:firstorderconv}
We now establish the main convergence result of the unique solution trajectory of \eqref{eq:eq1}.
	\begin{theorem}\label{thm:1st_order_convergence}
		Let $x$ be the unique solution trajectory  of \eqref{eq:eq1} under \cref{assump:1,assump:2,assump:5}, and suppose that \cref{assump:4} also holds. Then, for all $t\geq t_0$, we have
		\begin{equation}\label{eq:speed1}
			\Vert x(t) - \bar{x}\Vert \leq \theta^{-1}\Big(2t - \hat{t}\Big),
		\end{equation}
		where $\theta$ is defined in \cref{lemma:M} and $\hat{t} = 2t_0 - \theta\left(\Vert x_0 - \bar{x}\Vert\right)$.
		
		Furthermore, if \cref{assump:3} holds instead and $\rho:=\frac{\beta\tau}{\mu} < 1$, we have
		\begin{equation}\label{eq:speed2}
			\Vert x(t) - \bar{x}\Vert \leq  C e^{-2\mu(1 - \rho)t}~~\mbox{for all}~t\geq t_0,
		\end{equation}
		with $C = \Vert x_0 - \bar{x}\Vert e^{2\mu(1 - \rho)t_0}$.
	\end{theorem}
	\begin{proof}
		 We observe thanks to \cref{prop:1st_order_wp} that
		\[
		\frac{1}{2}\frac{\dd}{\dd t}\Vert x(t) - \bar{x}\Vert^{2} = \langle \dot{x}(t),x(t)- \bar{x} \rangle.
		\]
		We then have, for any selection $u(t)$ of $F_{\mm_{\bar{x}}}(x(t))$
		\begin{align}
			\frac{1}{2}\frac{\dd}{\dd t}\Vert x(t) - y\Vert^{2} 
			&= \langle \dot{x}(t),x(t)- \bar{x} \rangle = - \langle u(t),x(t)-\bar{x}\rangle - \langle \e_{\bar{x}}(x(t)),x(t)-\bar{x}\rangle \nonumber\\
			&= \langle u(t) - 0 ,x(t) - \bar{x}\rangle - \langle \e_{\bar{x}}(x(t)),x(t)-\bar{x}\rangle . \label{prop_proof:eq1}
		\end{align}
		We then get by \cref{assump:4}
		\[
		\langle u(t) - 0 ,x(t) - \bar{x}\rangle \geq  \Vert x(t) - \bar{x}\Vert\phi\left( \Vert x(t) - \bar{x}\Vert\right).
		\]
		Then, thanks to \cref{corollary:1}, \eqref{prop_proof:eq1} gives
		\begin{equation}\label{eq:der}
			\begin{aligned}
				\frac{1}{2}\frac{\dd}{\dd t}\Vert x(t) - \bar{x}\Vert^{2} 
				&\leq - \Vert x(t) - \bar{x}\Vert\phi\left(\Vert x(t) - \bar{x}\Vert\right)  + \Vert \e_{\bar{x}}(x(t))\Vert \Vert x(t) - \bar{x}\Vert \\
				&\leq   - \Vert x(t) - \bar{x}\Vert\phi\left(\Vert x(t) - \bar{x}\Vert\right)  + \beta \tau  \Vert x(t) - \bar{x}\Vert^{2}\\
				&= - \Vert x(t) - \bar{x}\Vert \varphi\left(\Vert x(t) - \bar{x}\Vert\right),
			\end{aligned}
		\end{equation}
		where $\varphi(t) = \phi(t) - \beta\tau t$ is as introduced in \eqref{eq:phi}. When $h(t):=\Vert x(t) - \bar{x}\Vert = 0$, then there is nothing to prove and we are done. Otherwise, $h(t) \neq 0$. Thus dividing both sides of \eqref{eq:der} by $h(t)$, we get
		\begin{equation}\label{eq:der1}
			\dot{h}(t)\leq -2\varphi\left( h(t)\right).
		\end{equation}
		From \eqref{eq:der} and \eqref{eq:der1}, we infer that
		\begin{equation}\label{eq:equation_M}
			\frac{\dd}{\dd t} \theta\left(h(t)\right) = \frac{-1}{\varphi\left(h(t)\right)} \dot{h}(t) \geq 2 .
		\end{equation}
		Integrating \eqref{eq:equation_M} between $t_0$ and $t>0$, we get
		\[
		\theta(h(t)) - \theta(h(t_0))\geq 2\left(t-t_0\right).
		\]
		Using \cref{lemma:M}, we deduce that
		\begin{equation}\label{eq:estimate1}
			h(t) \leq \theta^{-1}(2t - \hat{t}),
		\end{equation}
		where $\hat{t} = 2t_0 - \theta(h(t_0))$. This proves \eqref{eq:speed1}, as desired.
		
		To prove the second claim, we observe that if $F_{\mm_{\bar{x}}}$ is $\mu$-strongly monotone, $\phi(t) = \mu t$ so that $\varphi(t) = (\mu - \beta\tau)t$. Since $\mu > \beta\tau$, we get $\theta(s) = \frac{-\log(s)}{\mu - \beta\tau}$. So that $\theta^{-1}(s) = e^{-(\mu-\beta\tau)s}$. Plugging this into \eqref{eq:estimate1}, we obtain 
		\begin{equation}\label{eq:estimate2}
			h(t) \leq 	\Vert x_0 - \bar{x}\Vert e^{-2(\mu - \beta\tau)(t-t_0)},
		\end{equation}
		as desired.		
		\end{proof}
		\begin{remark}
		Observe that one can obtain \eqref{eq:speed2} from \eqref{eq:der} using Gronwall's \cref{lemma:gronwall1}. Indeed, taking $\phi(t) = \beta\mu t$, we infer from \eqref{eq:der}
		\[
		\frac{1}{2}\frac{\dd}{\dd t}\Vert x(t) - \bar{x}\Vert^{2} \leq -(\mu - \beta\tau)\Vert x(t) - \bar{x}\Vert^{2}.
		\]
	\end{remark}

		By \cref{assump:2} and \cref{thm:1st_order_convergence}, we obtain a rate of convergence of the measure $\mm_{x(t)}$ to $\mm_{\bar{x}}$ in the $\Wass_1$ distance as $t\to\infty$.
	\begin{corollary}\label{cor:W1_rate1}
		Let $x:[t_0,\infty)\to\R$ be the solution of \eqref{eq:eq1} under the assumptions of \cref{prop:1st_order_wp}, and suppose that \cref{assump:4} also holds. We have for all $t\geq t_0$
		\[
		\Wass_{1}(\mm_{x(t)},\mm_{\bar{x}}) \leq \tau  \theta^{-1}\Big(2t - \hat{t}\Big),
		\]
		where $\theta$ is defined in \cref{lemma:M} and $\hat{t} = 2t_0 - \theta\left(\Vert x_0 - \bar{x}\Vert\right)$. If \cref{assump:3} holds instead and $\rho:=\frac{\beta\tau}{\mu} < 1$, we have
			\[
			\Wass_{1}(\mm_{x(t)},\mm_{\bar{x}}) \leq C e^{-2\mu(1-\rho)t},
			\]
				with $C =  \tau\Vert x_0 - \bar{x}\Vert e^{2\mu(1 - \rho)t_0}.$
	\end{corollary}
\section{ Inertial second-order system with viscous and Hessian damping}\label{section:hessian2ndorder}
 In this section, to avoid technicalities, we restrict ourselves to the smooth convex optimization case where \ie $A = \nabla g$ so that (recall \eqref{eq:f_mu})
\begin{equation}\label{eq:operator_smooth_case}
F_{\mm} = \nabla G_{\mm}~~\mbox{with}~~G_{\mm} := g + f_{\mm} ~\text{and}~f_\mm := \Ex_{\xi\sim\mm}(f(\cdot,\xi)).
\end{equation}
In this case, \cref{assump:1,assump:5} read as follows: 
\begin{assumption}\label{assump:1'}{~}\medskip
\begin{itemize}
\item $g \in C^1(\cH) \cap \Gamma_0(\cH)$;
\item $f \in C^1(\cH \times \Xi)$, $f(x,\cdot)$ is $\mm_x$-measurable and $C^{1,1}_\beta(\Xi)$ for every $x \in \cH$, and $f_{\mm_{x}} \in C^{1,1}_L(\cH)$ for every $x \in \cH$.
\end{itemize}
\end{assumption}
On the other hand, \cref{assump:3} specializes to
\begin{assumption}[Strong convexity]\label{assump:3'}
$\exists \mu > 0$ such that $G_{\mm_{x}}$ is $\mu$-strongly convex for every $x\in\cH$.
\end{assumption}
This last assumption is true for instance if $f(\cdot,\xi)$ is $\mu$-strongly convex for $m_x$-almost
every $\xi \in \Xi$.\\
 
Let us first start with a discussion of the following second-order in time ODE
	\begin{equation}
		\label{eq:sec2_1}
		\left\lbrace
		\begin{aligned}
			&\ddot{x}(t) + \gamma(t)\dot{x}(t) + \nabla G_{\mm_{x(t)}}(x(t)) = 0,~\text{a.e}~t\in [t_0,T]\\
			&x(t_0) = x_0 \in \cH,
		\end{aligned}
		\right.
	\end{equation}
	where $\gamma:\R^{+}\to\R^{+}$ is a continuous function, usually referred to as the viscous damping coefficient. In a smooth convex optimization deterministic setting, \ie $F_{\mm_{x}} = \nabla f$ where $f$ is a smooth, strongly convex function, systems of the form \eqref{eq:sec2_1} were first studies by Polyak in \cite{polyak1964some} with a fixed viscous damping coefficient. Later on, this kind of systems were studied by \cite{attouch2000heavy} and then \cite{Su&al} where the authors establish the link between the continuous dynamics with $\gamma(t) = \frac{3}{t}$ and the Nesterov's gradient method \cite{Nesterov1983AMF}. This is a very active research field (see, \eg \cite{attouch2016rate,apidopoulos2020convergence,attouch2019rate,laszlo2021convergence} and the references therein). \\
	
Following what we have done \cref{section:firstorder}, \eqref{eq:sec2_1} can be reformulated again as a Lipschitzian perturbation of an inertial gradient system as follows
\begin{equation}\label{eq:pertubed_HBF}
		\ddot{x}(t) + \gamma(t)\dot{x}(t) + \nabla G_{\mm_{\bar{x}}}(x(t)) +  \e_{\bar{x}}(x(t)) = 0 .
\end{equation}
Even in absence of perturbations, it is well-known known that the inertial system \eqref{eq:pertubed_HBF} may suffer from transverse oscillations, and it is desirable to attenuate them. This is precisely the motivation behind the introduction of geometric Hessian-driven damping in \cite{alvarez2002second} (sometimes called Newton-type inertial dynamics). The inertial system we consider then features both viscous and Hessian-driven damping and reads
\begin{equation*}
\ddot{x}(t) + \gamma(t)\dot{x}(t) + \omega(t) \frac{\dd }{\dd t}\left(\nabla G_{\mm_{\bar{x}}}(x(t))\right) + \nabla G_{\mm_{x(t)}}(x(t)) = 0,
\end{equation*}
where $\omega:[0,\infty)\to\R^{+}$ is a continuous function usually referred to as the Hessian-driven damping coefficient, and which will be taken to be constant, \ie $\omega(t)\equiv \omega >0$. The case $\omega=0$ then recovers \eqref{eq:sec2_1}. Following the reasoning of \cref{section:firstorder}, then under \cref{assump:1',assump:2,assump:3'} this system is also equivalent to
		\begin{equation}
		\tag{ISEHD$_{\gamma,\omega}$}
		\label{eq:sec2_explicit}
		\ddot{x}(t) + \gamma(t)\dot{x}(t)  + \nabla G_{\mm_{\bar{x}}}(x(t)) + \omega \frac{\dd }{\dd t}\left(\nabla G_{\mm_{\bar{x}}}(x(t))\right) +  \e_{\bar{x}}(x(t)) +{\omega\frac{\dd }{\dd t}\e_{\bar{x}}(x(t))} = 0~\text{a.e}~t\in [t_0,T] .
	\end{equation}
	Recall that ${\e}_{\bar{x}}$ is Lipschitz continuous by Lemma~\ref{lemma:3}, and thus a.e. differentiable so that the last term in system \eqref{eq:sec2_explicit} makes sense whenever $x$ is absolutely continuous.
	
The acronym ISEHD stands for Inertial System with Explicit Hessian Damping. The Hessian damping is said to be explicit since, when $G_{\mm_{\bar{x}}}$ is of class $C^2$ and $x$ is smooth, we have 
	\[
	\frac{\dd }{\dd t}\left(\nabla G_{\mm_{\bar{x}}}(x(t))\right) = \nabla^2 G_{\mm_{\bar{x}}}(x(t)) \dot{x}(t).
	\]
Variants and generalizations of the inertial systems with Hessian-driven damping were studied by multiple authors (see, \eg, \cite{attouch2014dynamical,attouch2016fast,Attouch-Chbani-Fadili-Riahi}). The study of the effect of perturbations on these systems, \ie problems of the form \eqref{eq:sec2_explicit}, was carried out in the recent work \cite{Attouch-Fadili-Kungertsev}. We will take inspiration from this work but our analysis will depart from that of \cite{Attouch-Fadili-Kungertsev} in some important ways. For instance, and most importantly, the perturbations in \cite{Attouch-Fadili-Kungertsev} depend only on time while they depend on the trajectory in this paper. This poses a few challenges that we tackle by exploiting the Lipschitz continuity property enjoyed by our perturbations (see for instance \cref{lemma:3}).

\subsection{Well-posedness}
	\subsubsection{Equivalent first-order formulation}
	\begin{proposition}\label{proposition:1st_order_form}
		Suppose that \cref{assump:1'} holds, that $\gamma(t)\geq 0, \omega >0$ with $\gamma \in C^1([t_0,+\infty[)$. For any initial conditions $(x_0,v_0)\in\cH\times \cH$, the dynamics \eqref{eq:sec2_explicit} admits an equivalent formulation of the form
		\begin{equation}
			\label{eq:system_explicit}
			\left\lbrace
			\begin{aligned}
				\dot{x}(t) &+ \omega\Big( \nabla G_{\mm_{\bar{x}}}(x(t)) + \e_{\bar{x}}(x(t)) \Big) - \left(\frac{1}{\omega} - \gamma(t)\right) x(t) + \frac{1}{\omega} y(t) &= 0\\
				\dot{y}(t)& - \left(\frac{1}{\omega} - \gamma(t) - \dot{\gamma}(t)\omega\right) x(t) + \frac{1}{\omega} y(t) &=0,
			\end{aligned}
			\right.
		\end{equation}
		with initial conditions $x(t_0) = x_0, y(t_0) = -\omega\left(v_0 + \omega\nabla G_{\mm_{\bar{x}}}(x_0)\right) + \left(1 - \omega\gamma(t_0)\right)x_0 - \omega^2 \e_{\bar{x}}(x_0)$.
	\end{proposition}
	
	\begin{proof}
		Let $(x,y)$ be a solution of \eqref{eq:system_explicit}. By differentiation of the first equation in \eqref{eq:system_explicit}, we get
		\[
		\ddot{x}(t) + \omega\frac{\dd }{\dd t}(\nabla G_{\mm_{\bar{x}}}(x(t))) + \omega\frac{\dd }{\dd t}\e_{\bar{x}}(x(t)) + \dot{\gamma}(t) x(t) -  \Big(\frac{1}{\omega} - \gamma(t)\Big) \dot{x}(t) +  \frac{1}{\omega} \dot{y}(t) = 0.
		\]
		Replacing $\dot{y}$ by its expression from the second equation in \eqref{eq:system_explicit}, we obtain:
		\begin{multline*}
			\ddot{x}(t) + \omega  \frac{\dd }{\dd t}(\nabla G_{\mm_{\bar{x}}}(x(t))) + \omega\frac{\dd }{\dd t}\e_{\bar{x}}(x(t)) + \dot{\gamma}(t) x(t) \\
			- \Big(\frac{1}{\omega} - \gamma(t)\Big) \dot{x}(t) 
			+ \frac{1}{\omega} \left(\Big(\frac{1}{\omega} - \gamma(t) - \dot{\gamma}(t)\omega\Big) x(t) - \frac{1}{\omega} y(t) \right)= 0,
		\end{multline*}
		and using again the first equation in \eqref{eq:system_explicit} to eliminate $y(t)$, we get:
		\begin{multline*}
			\ddot{x}(t) + \omega  \frac{\dd }{\dd t}(\nabla G_{\mm_{\bar{x}}}(x(t))) + \omega \frac{\dd }{\dd t}\e_{\bar{x}}(x(t)) + \dot{\gamma}(t) x(t) - \Big(\frac{1}{\omega} - \gamma(t)\Big) \dot{x}(t) \\
			+ \frac{1}{\omega} \left(\Big(\frac{1}{\omega} - \gamma(t)-\dot{ \gamma}(t)\omega\Big) x(t)  + \dot{x}(t) + \omega\Big( \nabla G_{\mm_{\bar{x}}}x(t) + \e_{\bar{x}}(x(t)) \Big) - \Big(\frac{1}{\omega} - \gamma(t)\Big) x(t) \right)= 0,
		\end{multline*}
		and after simplifications, we recover \eqref{eq:sec2_explicit}. Conversely, let $x$ be a trajectory solution to \eqref{eq:sec2_explicit} with initial conditions $(x_0,v_0)\in\cH\times \cH$ and define 
		\[
		y(t) = -\omega\left(\dot{x}(t) + \omega\Big( \nabla G_{\mm_{\bar{x}}}x(t) + \e_{\bar{x}}(x(t)) \Big) - \Big(\frac{1}{\omega} - \gamma(t)\Big)x(t) \right).
		\]
		By differentiating the previous formula and using \eqref{eq:sec2_explicit}, we recover the second equation of \eqref{eq:system_explicit}, as desired.
	\end{proof}

 In the above, we have taken the derivatives in time as if they exist. We will show shortly that this is actually the case.

\subsubsection{Existence and uniqueness of a solution trajectory}
 \cref{proposition:1st_order_form} opens the door to  considering even the nonsmooth version of \eqref{eq:sec2_explicit} via \eqref{eq:system_explicit} requiring only that $g \in \Gamma_{0}(\cH)$. Thus, writing $\partial G_{\mm_{\bar{x}}}(x(t)) = \partial g + \Ex_{\xi\sim\mm_{\bar{x}}}(\nabla f(\cdot,\xi))$, we can consider the differential inclusion

		\begin{equation}
			\label{eq:system_explicit_inclusion}
			\left\lbrace
			\begin{aligned}
				\dot{x}(t) &+ \omega\Big( \partial G_{\mm_{\bar{x}}}(x(t)) + \e_{\bar{x}}(x(t)) \Big) - \left(\frac{1}{\omega} - \gamma(t)\right) x(t) + \frac{1}{\omega} y(t) &\ni 0\\
				\dot{y}(t)& - \left(\frac{1}{\omega} - \gamma(t) - \dot{\gamma}(t)\omega\right) x(t) + \frac{1}{\omega} y(t) &=0,
			\end{aligned}
			\right.
		\end{equation}
		with initial conditions $(x(t_0) = x_0, y(t_0) = y_0) \in {\dom(G_{\mm_{\bar{x}}})}\times \cH$.

One of the main advantages of \eqref{eq:system_explicit_inclusion} is that it can be easily recast as a differential inclusion governed by a Lipschitz perturbation of a maximal monotone operator on the product space $\cH\times\cH$. Indeed, setting $Z(t)= (x(t),y(t))$, $\Aa(x,y) = (\omega \partial G_{\mm_{\bar{x}}}(x),0)$ and 
	\[
	\E(t,x,y) = \left(\omega \e_{\bar{x}}(x(t)) - \Big(\frac{1}{\omega} - \gamma(t)\Big) x(t) + \frac{1}{\omega} y(t) ,  -\Big(\frac{1}{\omega} - \gamma(t) - \dot{\gamma}(t)\omega\Big) x(t) + \frac{1}{\omega} y(t)\right),
	\]
	we immediately see that \eqref{eq:system_explicit_inclusion} can be written as
	\begin{equation}
		\dot{Z}(t) +  \Aa(Z(t)) + \E(t,Z(t))\ni 0_{\cH\times\cH},~Z(0) = (x_0,y_0).
	\end{equation}
	 Observe that $\Aa$ is nothing but the subdifferential of the function $H(x,y)=\omega G_{\mm_{\bar{x}}}(x)$, and thus $\Aa$ is maximal monotone if the smooth part is monotone (see \cref{lemma:0}). Then, \eqref{eq:system_explicit_inclusion} can be written as the differential inclusion in the product space $\cH\times\cH$ 
	\begin{equation}
		\tag{MIS}
		\label{eq:sec2_3}
		\left\lbrace
		\begin{aligned}
			\dot{Z}(t) &+  \Aa(Z(t)) + \E(t,Z(t))\ni 0_{\cH\times\cH},~\text{a.e}~t\in [t_0,T]\\
			Z(0) &=(x_0,y_0) \in \dom(g)\times \cH,
		\end{aligned}
		\right.
	\end{equation}
	which fits in the framework of Lipschitz perturbations of maximal monotone operators as in \cref{section:firstorder}. Notice that this formulation is different from the classical Hamiltonian one.
	
	Before stating the main result, let us recall that we endow the product space with the scalar product $\langle(u,v),(u^*,v^*)\rangle_{\cH\times\cH} = \langle u,u^*\rangle + \langle v,v^*\rangle$, and the induced norm $\Vert (u,v)\Vert_{\cH\times\cH} = \sqrt{\Vert u\Vert^2 + \Vert v\Vert^2}$.
	We have the following auxiliary results
	\begin{lemma}\label{lemma:lip_2nd_order}
		 Suppose that \cref{assump:1',assump:2} hold. Consider the operator $\E$ where $\omega >0$ and $\gamma \in C^1([t_0,+\infty[)$. Then, for any $t \in [t_0,T]$, $\E(t,\cdot,\cdot)$ is Lipschitz continuous on $\cH \times \cH$.
	\end{lemma}
The dependence on $t$ of the Lipschitz constant appears only through $|\gamma(t)|$ and $|\dot{\gamma}(t)|$. With the two most popular choices of $\gamma(t)$, constant (as in heavy ball method) or the asymptotically vanishing viscous damping $\gamma(t)=\alpha/t$, both $|\gamma(t)|$ and $|\dot{\gamma}(t)|$ are uniformly bounded. This makes $\E(t,\cdot,\cdot)$ Lipschitz continuous uniformly in $t > t_0$.
	\begin{proof}
		Let $u,v,u^*,v^*\in\cH$ and set $p = (\frac{1}{\omega} - \gamma(t)), q = (\frac{1}{\omega} - \gamma(t)-\dot{\gamma}(t)\omega).$ We have, for $t\in [t_0,T]$
		\[
		\begin{aligned}
			&\Vert \E(t,u,v) - \E(t,u^*,v^*)\Vert_{\cH\times\cH } \\
			&= \Big\Vert \Big(\omega(\e_{\bar{x}}(u)-\e_{\bar{x}}(u^*) + p(u^*-u)+\frac{1}{\omega}(v-v^*),\frac{1}{\omega}(v-v^*) + q(u^*-u)\Big)\Big\Vert_{\cH\times\cH}\\
			&=\sqrt{\Vert(\omega(\e_{\bar{x}}(u)-\e_{\bar{x}}(u^*) + p(u^*-u)+\frac{1}{\omega}(v-v^*)\Vert^2 + \Vert\frac{1}{\omega}(v-v^*) + q(u^*-u)\Vert^2}\\
			&\leq\sqrt{2\omega^2\Vert\e_{\bar{x}}(u)- \e_{\bar{x}}(u^*)\Vert^2 + (4p^2+2q^2)\Vert u-u^*\Vert^2 + \frac{6}{\omega^2}\Vert v-v^*\Vert^2}\\
			&=\sqrt{(2\omega^2  (2L+\beta\tau)^2 + 4p^2+2q^2)\Vert u-u^*\Vert^2 + \frac{6}{\omega^2}\Vert v-v^*\Vert^2}\\
					&\leq \Big(\sqrt{2}\omega (2L+\beta\tau) + 2\vert p\vert+\sqrt{2}\vert q\vert + \frac{\sqrt{6}}{\omega}\Big)\Vert (u,v) - (u^*,v^*)\Vert_{\cH\times\cH}\\
			&= K(\omega,\beta,\tau,\gamma,t)\Vert (u,v) - (u^*,v^*)\Vert_{\cH\times\cH}\\
		\end{aligned}
		\]
		where we have used Young's inequality and \cref{lemma:3} for the Lipschitz continuity of the operator $\e_{\bar{x}}$.
	\end{proof}

\begin{proposition}\label{prop:existence}
Assume that \cref{assump:1',assump:2,assump:3'} hold where we only require $g \in \Gamma_0(\cH)$ (but possibly nonsmooth). Suppose also that $\omega >0$ and $\gamma \in C^1([t_0,+\infty[;\cH)$ such that both $\gamma$ and $\dot{\gamma} \in L^2([t_0,T])$ for all $T > t_0$. Then, for any initial data $x_0\in\dom(g)$ and $y_0\in\cH$, there exists a unique global strong solution $(x,y): [t_0,+\infty[ \to \cH \times \cH$ to \eqref{eq:sec2_3} such that $x(t_0) = x_0$ and $y(t_0) = y_0$. Moreover, the solution $Z = (x,y)$ satisfies the following properties
		\begin{enumerate}[(i)]
			\item\label{it:1} $y \in C^1([t_0,+\infty[;\cH)$, and $\dot{y}(t)-\left(\frac{1}{\omega}-\gamma(t)-\omega\dot{\gamma}(t)\right) x(t)+\frac{1}{\omega} y(t)=0$, for all $t \geq t_0$;
\item\label{it:2} $x$ is absolutely continuous on $\left[t_0, T\right]$ and $\dot{x} \in L^2\left(t_0, T ; \mathcal{H}\right)$ for all $T>t_0$; 
\item\label{it:3}  $x(t) \in \dom(\partial g)$ for all $t>t_0$; 
\item\label{it:4} $x$ is Lipschitz continuous on any compact subinterval of $] t_0,+\infty[;$
\item\label{it:5} the function $t\in \left[t_0,+\infty\left[ \mapsto G_{\mm_{\bar{x}}}(x(t))\right.\right.$ is absolutely continuous on $\left[t_0, T\right]$ for all $T>t_0$; 
\item\label{it:6} there exists a function $\xi:\left[t_0,+\infty[\rightarrow \cH\right.$ such that
\begin{enumerate}
\item $\xi(t) \in \partial G_{\mm_{\bar{x}}}(x(t))$ for all $t>t_0$;
\item $\dot{x}(t)+\omega \xi(t)-\Big(\frac{1}{\omega}-\gamma(t)\Big) x(t)+\frac{1}{\omega} y(t)=0$ for almost every $t>t_0$;
\item  $\xi \in L^2\left(t_0, T ; \cH\right)$ for all $T>t_0$;
\item $\frac{\dd }{\dd t} G_{\mm_{\bar{x}}}(x(t))=\langle\xi(t), \dot{x}(t)\rangle$ for almost every $t>t_0$.
\end{enumerate}
\end{enumerate}
\end{proposition}

\begin{proof}
In view of our assumptions, $\Aa$ is maximal monotone (see \cref{lemma:0}) and $\E$ is Lipschitz continuous thanks to \cref{lemma:lip_2nd_order}. We deduce the existence of a unique strong global solution $Z=(x,y):[t_0,T]\to\cH\times\cH$ of \eqref{eq:sec2_3} via \cite[Proposition~3.12]{brezismaxmon}. The verification of items \ref{it:3}-\ref{it:6} can be done by following the main arguments of \cite[Theorem~4.4]{attouch2016fast}
	\end{proof}

Assuming that $g \in C^1(\cH)$ with a Lipschitz continuous gradient, the conclusions of \cref{prop:existence} can be strengthened as follows.
\begin{proposition}\label{prop:existencesmooth}
Assume that \cref{assump:1',assump:2,assump:3'} hold with $\nabla g$ Lipschitz continuous. Suppose also that $\omega >0$ and $\gamma \in C^1([t_0,+\infty[;\cH)$ such that both $\gamma$ and $\dot{\gamma} \in L^2([t_0,T])$ for all $T > t_0$. Then, for any $t_0 > 0$, and any Cauchy data $(x_0, \dot{x}_0)$, the system \eqref{eq:sec2_explicit} admits a unique global solution $x \in C^1([t_0,+\infty[;\cH)$ satisfying $(x(t_0),\dot{x}(t_0)) = (x_0, \dot{x}_0)$. Moreover, $\dot{x}$, $\nabla G_{\mm_{\bar{x}}}(x(\cdot))$ and $\e_{\bar{x}}(x(\cdot))$  are absolutely continuous.
\end{proposition}

\begin{proof}
Under our regularity assumptions, we get from the conclusion of \cref{prop:existence} and the first equation of \eqref{eq:system_explicit} that $x$ is a $C^1([t_0, +\infty[;\cH)$ function. Moreover, $\dot{x}$ is absolutely continuous thanks to Lipschitz continuity of $\nabla G_{\mm_{\bar{x}}}$ and $\e_{\bar{x}}$, absolute continuity of $x$ already established in \cref{prop:existence}, together with continuity of $\gamma$, $\dot{\gamma}$ and $y$.
\end{proof}

\begin{remark}
The lack of differentiability of $\e_{\bar{x}}$ is the main obstacle for showing the existence of a classical solution even if $G_{\mm_{\bar{x}}}$ is assumed $C^2$. Moreover, arguing using the non-autonomous Cauchy-Lipschitz theorem, the convexity assumption on $g$ can be dropped in \cref{prop:existencesmooth} at the price of assuming $\gamma$ to be twice continuously differentiable on $[t_0,+\infty[$. The proof is left to the reader.
\end{remark}

\subsection{Convergence properties}
Now let us examine the convergence properties of \eqref{eq:sec2_explicit}. 
We will work under \cref{assump:1',assump:2,assump:3'} where $\nabla g$ is Lipschitz continuous. The $\mu$-strong convexity of $G_{\mm_{\bar{x}}}$ allows to tune the viscous damping coefficient to the modulus $\mu$ by taking $\gamma(t)\equiv  2\sqrt{\mu}$ as advocated in \cite{Attouch-Chbani-Fadili-Riahi,Attouch-Fadili-Kungertsev}. From now on, we focus on the following system
		\begin{equation}
		\tag{ISEHD$_{2\sqrt{\mu},\omega}$}
		\label{eq:sec2_explicit_st_cvx}
		\ddot{x}(t) + {2\sqrt{\mu}}\dot{x}(t)  + \nabla G_{\mm_{\bar{x}}}(x(t)) + \omega  \frac{\dd }{\dd t}\left(\nabla  G_{\mm_{\bar{x}}}(x(t))\right) +  \e_{\bar{x}}(x(t))+{\omega\frac{\dd }{\dd t}\e_{\bar{x}}(x(t))} = 0 .
	\end{equation}
	Thanks to our assumptions, this systems falls within the framework of the previous section so that \cref{prop:existencesmooth} applies.
	
	 To perform Lyapunov analysis, let us define the following energy function $\V:[t_0,\infty[\to\R^+$ by
	\[
	\V(t):= G_{\mm_{\bar{x}}}(x(t)) - G_{\mm_{\bar{x}}}^{*} + \frac{1}{2}\Vert v(t)\Vert^{2},~\mbox{where}~v(t) := \sqrt\mu\left(x(t) - \bar{x}\right) + \dot{x}(t) + \omega \nabla G_{\mm_{\bar{x}}}(x(t)).
	\]
	 where $G_{\mm_{\bar{x}}}^{*} = \min G_{\mm_{\bar{x}}}(\cH)$.\\
		
We prove the following result.
	\begin{theorem}\label{thm:th1}
		 Assume that \cref{assump:1',assump:2,assump:3'} hold and $\nabla g$ is Lipschitz continuous. Let $x:[t_0,+\infty[\to\cH$ be the solution trajectory of \eqref{eq:sec2_explicit_st_cvx}. Suppose that $\rho := \frac{\beta\tau}{\mu}$ and the damping coefficient $\omega$ satisfy  
		\begin{equation}\label{eq:condition_rho_omega}
		 0 \leq \omega\leq\min\left(\frac{1}{2\sqrt\mu},\frac{\sqrt{\mu}}{2\sqrt{2} (2L+\beta\tau)}\right) \quad \text{and} \quad 16\rho^2 + \omega < 1.
		\end{equation} 
		Then, we have:
		\begin{enumerate}[label=(\arabic*)]
		\item\label[noref]{it:hd-it1} for all $t\geq t_0$
		\[
		\V(t)\leq \V(t_0) e^{-\frac{\sqrt\mu}{4}(t-t_0)} .
		\]
		In particular
\[
\frac{\mu}{2}\Vert x(t)-\bar{x}\Vert^2 \leq G_{\mm_{\bar{x}}}(x(t)) - G^{*}_{\mm_{\bar x}} \leq \V(t_0) e^{-\frac{\sqrt\mu}{4}(t-t_0)}.
\]		\item\label[noref]{it:hd-it2} There exists $C>0$ such that,
	\[
e^{-\sqrt\mu t}\int_{t_0}^{t} e^{\sqrt{\mu}s}\Vert\nabla G_{\mm_{\bar{x}}}(x(s))\Vert^2 \dd s \leq C e^{-\frac{\sqrt{\mu}}{4}t},~\forall t\geq t_0.
\]
		\end{enumerate}
	\end{theorem}
	
	 Some remarks are in order before proving this result.
	\begin{remark}{~}\medskip
	\begin{itemize}
	\item Compared to the convergence results in \cref{thm:1st_order_convergence} on the first-order system \eqref{eq:eq1}, \cref{thm:th1} not only provides a fast convergence rate of the trajectory and objective value, but also of the gradient. Observe also the $\sqrt{\mu}$ in the rate in \cref{thm:th1} instead of $\mu$ \cref{thm:1st_order_convergence}. This recovers the known result that the system \eqref{eq:sec2_explicit_st_cvx} is faster than \eqref{eq:eq1} for badly conditioned objectives, approaching the optimal rate of the class of strongly convex functions with Lipschitz continuous gradient.
	
	\item Contrary to \cite{Attouch-Fadili-Kungertsev}, no assumption on the integrability of $\e_{\bar{x}}(x(\cdot))$ and $\frac{\dd }{\dd t}{\e}_{\bar{x}}(x(\cdot))$ is needed in \cref{thm:th1}. In fact, thanks to \cref{corollary:1} and \cref{lemma:3}, the norms of these error terms will be absorbed in the right hand side of \eqref{eq:V1}.

	\item In case $\omega = 0$, \ie when the inertial dynamic is considered only with the viscous damping coefficient (which can be view as a perturbed HBF method), the condition \eqref{eq:condition_rho_omega} reduces to $ \rho<\frac{1}{4}$. This contrasts with the convergence condition for first-order dynamics (\cf \cref{thm:1st_order_convergence}), where convergence is ensured in  parameter regime $\rho<1$. One possible explanation for this difference, is the potential occurrence of oscillations typical of inertial system, which may necessitate a stricter compatibility condition between the parameters $\tau,\beta,\mu$.	
		\end{itemize}
		\end{remark}

	\begin{proof}
		 In view of our assumptions, \cref{prop:existencesmooth} applies and we get that $v: [t_0,+\infty[ \to \cH$ is absolutely continuous and thus so is $\V$. We then have
		\begin{equation}
			\begin{aligned}
				\dot{\V}(t) &= \langle\nabla G_{\mm_{\bar{x}}}(x(t)),\dot{x}(t)\rangle + \langle v(t),\sqrt{\mu}\dot{x}(t) + \ddot{x}(t) +\omega\frac{\dd }{\dd t}\left(\nabla  G_{\mm_{\bar{x}}}(x(t))\right) \rangle\\
				& = \langle\nabla G_{\mm_{\bar{x}}}(x(t)),\dot{x}(t)\rangle+\langle v(t),-\sqrt{\mu}\dot{x}(t) - \nabla G_{\mm_{\bar{x}}}(x(t)) - \e_{\bar{x}}(x(t)) -\omega\frac{\dd }{\dd t}\e_{\bar{x}}(x(t))  \rangle ,
			\end{aligned}
		\end{equation}
		 where we used the constitutive equation in \eqref{eq:sec2_explicit_st_cvx}. Replacing $v$ by its expression and rearranging, we arrive at
		\begin{dmath}
				\dot{\V}(t) +\mu \langle\dot{x}(t),x(t)-\bar{x}\rangle +\sqrt{\mu}\Vert\dot{x}(t)\Vert^2 +\sqrt{\mu} \langle \nabla G_{\mm_{\bar{x}}}(x(t)),x(t)-\bar{x}\rangle + \omega\sqrt{\mu}\langle\nabla G_{\mm_{\bar{x}}}(x(t)),\dot{x}(t)\rangle + \omega\Vert\nabla G_{\mm_{\bar{x}}}(x(t))\Vert^2    = -  \langle v(t),\e_{\bar{x}}(x(t))+\omega\frac{\dd }{\dd t}\e_{\bar{x}}(x(t))\rangle.
				\label{eq:der_V}
		\end{dmath}
		Using $\mu$-strong convexity of $G_{\mm_{\bar{x}}}$, we have
		\begin{equation}
			\langle\nabla G_{\mm_{\bar{x}}}(x(t)), x(t) - \bar{x}\rangle \geq G_{\mm_{\bar{x}}}(x(t)) - G_{\mm_{\bar{x}}}^{*} + \frac{\mu}{2}\Vert x(t)-\bar{x}\Vert
			^{2},
		\end{equation}
		and plugging this into \eqref{eq:der_V}, we get
		\begin{equation}\label{eq:V1}
			\dot{\V}(t) + \sqrt{\mu} \Theta(t) \leq  \Vert v(t)\Vert\Vert\e_{\bar{x}}(x(t))+\omega\frac{\dd }{\dd t}\e_{\bar{x}}(x(t))\Vert,
		\end{equation}
		where 
		\begin{multline*}
		\Theta(t) := G_{\mm_{\bar{x}}}(x(t)) - G_{\mm_{\bar{x}}}^{*} + \frac{\mu}{2}\Vert x(t) - \bar{x}\Vert^2 + \sqrt\mu \langle\dot{x}(t),x(t)-\bar{x}\rangle +\Vert\dot{x}(t)\Vert^2 \\
		+ \omega\langle\nabla G_{\mm_{\bar{x}}}(x(t)),\dot{x}(t)\rangle + \frac{\omega}{\sqrt\mu}\Vert\nabla G_{\mm_{\bar{x}}}(x(t))\Vert^2.
		\end{multline*}
		Using the definition of $\V(t)$, we may rewrite $\Theta(t)$ as 
		
		\[
		\Theta(t) = \V(t) +\frac{1}{2}\Vert\dot{x}(t)\Vert^2 -\omega\sqrt{\mu}\langle\nabla G_{\mm_{\bar{x}}}(x(t)),x(t) -\bar{x}\rangle + \left( \frac{\omega}{\sqrt{\mu}} - \frac{\omega^2}{2} \right)\Vert\nabla G_{\mm_{\bar{x}}}(x(t))\Vert^2.
		\]
		Consequently, \eqref{eq:V1} becomes
		\begin{multline}
		\dot{\V}(t) + \sqrt{\mu} \V(t) + \frac{\sqrt{\mu}}{2}\Vert\dot{x}(t)\Vert^2 \\
		+  \sqrt{\mu}\Bigg( \left( \frac{\omega}{\sqrt{\mu}} - \frac{\omega^2}{2} \right)\Vert\nabla G_{\mm_{\bar{x}}}(x(t))\Vert^2 -\omega\sqrt{\mu}\langle\nabla G_{\mm_{\bar{x}}}(x(t)),x(t) -\bar{x}\rangle\Bigg) \\
		\leq \Vert v(t)\Vert\Vert\e_{\bar{x}}(x(t))+\omega\frac{\dd }{\dd t}\e_{\bar{x}}(x(t))\Vert.
		\end{multline}
		Using strong convexity of $G_{\mm_{\bar{x}}}$ again, and discarding the quadratic term in $v(t)$, we have
		\[
		\V(t) = \frac{1}{2}\V(t) + \frac{1}{2}\V(t) \geq \frac{1}{2}\V(t) + \frac{\mu}{4}\Vert x(t) - \bar{x}\Vert^2 .
		\]
		 Observing that $\frac{\omega}{2\sqrt\mu}\leq \frac{\omega}{\sqrt{\mu}} - \frac{\omega^2}{2}$ for $0\leq \omega \leq \frac{1}{\sqrt\mu}$, we end up with
		\begin{multline}
		\dot{\V}(t)+\frac{\sqrt\mu}{2}\V(t) + \frac{\sqrt\mu}{2}\Vert\dot{x}(t)\Vert^2 \\
		+ \sqrt{\mu}\Bigg(  \frac{\mu}{4}\Vert x(t) - \bar{x}\Vert^2 +\frac{\omega}{2\sqrt{\mu}} \Vert\nabla G_{\mm_{\bar{x}}}(x(t))\Vert^2 -\omega\sqrt{\mu}\Vert\nabla G_{\mm_{\bar{x}}}(x(t))\Vert\Vert x(t) -\bar{x}\Vert\Bigg) \\
		\leq \Vert v(t)\Vert\Vert\e_{\bar{x}}(x(t))+\omega\frac{\dd }{\dd t}\e_{\bar{x}}(x(t))\Vert.
		\end{multline}
		 Now let us treat the right hand side of this inequality. Thanks to \cref{corollary:1}, we have $\Vert \e_{\bar{x}}(x(t))\Vert = \Vert \e_{\bar{x}}(x(t)) - \e_{\bar{x}}(\bar{x})\Vert \leq \beta\tau\Vert x(t)-\bar{x}\Vert$, and by \cref{lemma:3} $\Vert\frac{\dd }{\dd t}\e_{\bar{x}}(x(t))\Vert \leq (2L+\beta\tau)\Vert\dot{x}(t)\Vert$. Since $\V(t)\geq \frac{1}{2}\Vert v(t)\Vert^2$, applying Young's inequality yields
		\begin{equation}
		\begin{aligned}
				\Vert v(t)\Vert\Vert\e_{\bar{x}}(x(t))+\omega\frac{\dd }{\dd t}\e_{\bar{x}}(x(t))\Vert &\leq \frac{\sqrt\mu}{8}\Vert  v(t)\Vert^2 + \frac{4}{\sqrt\mu}\Vert\e_{\bar{x}}(x(t))\Vert^2+\frac{4\omega^2}{\sqrt{\mu}}\Vert\frac{\dd }{\dd t}\e_{\bar{x}}(x(t))\Vert^2\\ &\leq \frac{\sqrt\mu}{4}\V(t) + \frac{4\beta^2\tau^2}{\sqrt\mu}\Vert x(t) - \bar{x}\Vert^2+\frac{4\omega^2(2L+\beta\tau)^2}{\sqrt\mu}\Vert \dot{x}(t)\Vert^2 \\
		\end{aligned}
		\end{equation}
		We get after rearranging the terms
		\begin{equation}
		\label{eq:V2}
			\dot{\V}(t) + \frac{\sqrt{\mu}}{4}\V(t) + \left(\frac{\sqrt\mu}{2}-\frac{4\omega^2(2L+\beta\tau)^2}{\sqrt\mu}\right)\Vert\dot{x}(t)\Vert^2 + \sqrt{\mu} \Psi(t) \leq 0,
		\end{equation}

where 
$$\Psi(t) = \left(\frac{\mu}{4} - \frac{4\beta^2\tau^2}{\mu}\right)\Vert x(t) - \bar{x}\Vert^2 + \frac{\omega}{2\sqrt{\mu}} \Vert\nabla G_{\mm_{\bar{x}}}(x(t))\Vert^2 -\omega\sqrt{\mu}\Vert\nabla G_{\mm_{\bar{x}}}(x(t))\Vert\Vert x(t) -\bar{x}\Vert. $$

		Setting 
		\[
		\BF{a} =  \frac{\mu}{4} - \frac{4\beta^2\tau^2}{\mu},  \BF{b} = \frac{\omega}{2\sqrt{\mu}}, \BF{c} = -\frac{\omega\sqrt\mu}{2}		
		\]
		and $X = \Vert x(t) - \bar{x}\Vert$ and $ Y  = \Vert\nabla G_{\mm_{\bar{x}}}(x(t))\Vert$, we see that $\Psi$ can be written as a quadratic form $\Q:\cH\times\cH\to\R$ with $\Q(X,Y) = \BF{a}\Vert X\Vert^2 + 2\BF{c}\langle X,Y\rangle + \BF{b}\Vert Y\Vert^2$. By assumption  $\BF{a,b}\geq 0$, and the discriminant $\BF{c}^2 - \BF{ab}$ of $\Q$ is nonpositive. Indeed, since $0\leq \omega\leq \frac{1}{2\sqrt\mu} $
	\[
	\BF{c}^2 - \BF{ab} = \frac{\omega^2\mu}{4} - \frac{\omega}{2\sqrt{\mu}} \left(\frac{\mu}{4} - \frac{4\beta^2\tau^2}{\mu}\right) \leq   \frac{\omega^2 \sqrt\mu}{8} - \frac{\omega\sqrt\mu}{2} \left(\frac{1}{4} - 4\rho^2\right) =  \frac{\omega \sqrt\mu}{8} (\omega - 1+16\rho^2) <0.
	\]
		
Hence, $\Psi(t) \geq 0$ and since $\omega\leq \frac{\sqrt{\mu}}{2\sqrt{2}(2L+\beta\tau)}$, we get $\left(\frac{\sqrt\mu}{2}-\frac{4\omega^2(2L+\beta\tau)^2}{\sqrt\mu}\right)\geq 0$, and thus
		\[
		\dot{\V}(t) + \frac{\sqrt\mu}{4}\V(t)  \leq 0,
		\]
		which gives after integration
		\begin{equation}\label{eq:estimate_V}
			\V(t)\leq \V(t_0) e^{\frac{-\sqrt\mu}{4}(t-t_0)}.
		\end{equation}
			Therefore, $\lim_{t\to\infty} \V(t) = 0$ and in particular
		\begin{equation}\label{eq:est1}
			\lim_{t\to\infty} G_{\mm_{\bar{x}}}(x(t)) - G_{\mm_{\bar{x}}}^{*}  = 0~~\mbox{and}~~\lim_{t\to\infty}\Vert v(t)\Vert = 0.
				\end{equation}
		This implies, using strong convexity of $G_{\mm_{\bar{x}}}$ 
		\[
		\lim_{t\to\infty}\Vert x(t)-\bar{x}\Vert = 0,
		\]
		which gives that 
		\[
		\lim_{t\to\infty} \Vert\nabla G_{\mm_{\bar{x}}}(x(t))\Vert = 0.
		\]
		We deduce from \eqref{eq:est1} that $\lim_{t\to\infty}\Vert\dot{x}(t)\Vert = 0$.

		Coming back to \eqref{eq:estimate_V}, we have, by definition of $\V$ and $\mu$-strong convexity of $G_{\mm_{\bar{x}}}$, that
		\begin{equation}\label{eq:est2}
		 \frac{\mu}{2}\|x(t)-\bar{x}\|^2 \leq G_{\mm_{\bar{x}}}(x(t)) - G_{\mm_{\bar{x}}}^{*} \leq  \V(t_0) e^{-\frac{\sqrt\mu}{4}(t-t_0)}~~\mbox{and}~~\Vert v(t)\Vert^2 \leq  2\V(t_0)  e^{-\frac{\sqrt\mu}{4}(t-t_0)}.
		\end{equation}
		Developing in \eqref{eq:est2} we have
		\begin{dmath}\label{eq:est3}
		\mu\Vert x(t) - \bar{x}\Vert^2 + \Vert\dot{x}(t)\Vert^2 + \omega^2\Vert\nabla G_{\mm_{\bar{x}}}(x(t))\Vert^2 + 2\omega\sqrt{\mu}\langle\nabla G_{\mm_{\bar{x}}}(x(t)), x(t) - \bar{x}\rangle \\+ 2\omega\langle\nabla G_{\mm_{\bar{x}}}(x(t)),\dot{x}(t)\rangle+2\sqrt{\mu}\langle\dot{x}(t),x(t)-\bar{x}\rangle \leq C e^{-\frac{\sqrt\mu}{4}t},
		\end{dmath}
		where  $C =  2\V(t_0) e^{\frac{\sqrt{\mu}}{4}t_0}.$ 
		
		Since $\langle\nabla G_{\mm_{\bar{x}}}(x(t)), x(t) - \bar{x}\rangle \geq G_{\mm_{\bar{x}}}(x(t)) - G_{\mm_{\bar{x}}}^{*}$, we deduce from \eqref{eq:est3} that
		\begin{equation}\label{eq:eq_U}
		\dot{U}(t)+\sqrt\mu U(t) + \omega^2\Vert\nabla G_{\mm_{\bar{x}}}(x(t))\Vert^2\leq  C e^{-\frac{\sqrt{\mu}}{4}t},
		\end{equation}
		where $U(t):=\sqrt{\mu}\Vert x(t)-\bar{x}\Vert^2 + 2\omega\left(G_{\mm_{\bar{x}}}(x(t)) - G_{\mm_{\bar{x}}}^{*}\right)$.
Integrating \eqref{eq:eq_U}, we obtain after elementary computation
\[
e^{-\sqrt\mu t}\int_{t_0}^{t} e^{\sqrt{\mu}s}\Vert\nabla G_{\mm_{\bar{x}}}(x(s))\Vert^2 \dd s \leq  C_{1} e^{-\frac{\sqrt{\mu}}{4}t},
\]
	as desired.
	\end{proof}
	
	We end this section with a similar result to \cref{cor:W1_rate1}, which is a direct consequence of \cref{assump:2} and \cref{thm:th1}\cref{it:hd-it1}.
	\begin{corollary}\label{cor:W1_rate2}
		Let $x:[t_0,+\infty[\to\cH$ be the solution of \eqref{eq:sec2_explicit_st_cvx} where \eqref{eq:condition_rho_omega} holds. Then $\forall t\geq t_0$ 
		\[
		\Wass_{1}(\mm_{x(t)},\mm_{\bar{x}}) \leq  \tau\sqrt{\frac{2}{\mu}\V(t_0)} e^{-\frac{\sqrt{\mu}}{8} (t-t_0)}. 
		\]
	\end{corollary}
\section{On coarse Ricci curvature}\label{section:ricci}
		In this section we discuss  some dynamical and geometrical properties of the family $(\mm_x)_{x \in \cH}$, particularly the notion of Ollivier-Ricci curvature and how it it tightly related to \cref{assump:2}. In fact, the family of probabilities $\mm = (\mm_{x})_{x\in\cH}$ and its Lipschitz behavior with respect to the $\Wass_1$-Wasserstein distance, reveals that a natural setting to address monotone inclusions of the form \eqref{eq:eq0}, and thus stochastic optimization problems with decision-dependent distributions is the framework of \textit{metric random walk spaces} (see, \eg \cite{OllivierY,mazon2023variational}). 	All definitions of this section can be found in \cite{mazon2023variational,HernandezLasserre}.
 
		\subsection{Metric random walk spaces}
	
	 Before going further, let us recall the following definitions to introduce a couple of probabilistic notions.
	\begin{definition}[Random walks \cite{OllivierY}]\label{def:RW}
		Given a Polish space $(X,\dd)$. A family of probabilities $\mm = (\mm_{x})_{x\in X}$ is a random walk on $X$ if $\mm_{x}\in\Pro$ for each $x\in X$ and
		\begin{itemize}
			\item $\mm_{x}$ depends measurably on $x\in X$,
			\item Each $\mm_{x}$ has finite first-order moment, \ie for some $x^{o}\in X$, $\Ex_{y\sim \mm_x} \dd( y,x^{o}) <\infty$.
		\end{itemize}
		Then $(X,\dd)$ equipped with a random walk $\mm$ is a metric random walk space (m.r.w.s for short), and we denote it by $[X,\dd,\mm]$.
	\end{definition}
	Let us recall the notion of invariant and ergodic measures.
	
	\begin{definition}[Invariance]
	Let $\nu$ be a $\sigma$-finite measure on $X$ and $\mm$ a random walk on $(X,\cB)$. We say that  $\nu$ is invariant with respect to $\mm$ if $\nu\star\mm = \nu$, where $\nu\star\mm$ is the convolution of $\nu$ by the random walk $\mm$ and is defined by
		\[
		\nu\star\mm(A) = \int_{X} \mm_{x}(A) \dd\nu(x)~\mbox{for all}~A\in\cB.
		\]
	\end{definition}

		As pointed out in \cite{OllivierY}, each measure $\mm_x$ can be seen as a replacement of a sphere around $x$. While in a probabilistic framework one think about a Markov chain whose transition kernel from $x$ to $y$ in $n$ steps is defined by
		\begin{equation}\label{eq:n-step}
		\dd\mm_{x}^{*n}(y) = \int_{z\in X} \dd\mm_{x}^{*(n-1)}(z)\dd\mm_{z}(y),
		\end{equation}
		with $\mm_{x}^{1} = \mm_x$ and $\mm_{x}^{0} = \delta_x$.\\

	In the sequel, we assume that $(\cH,\dd)$ is a separable real Hilbert space, and thus a Polish space, where $\dd(x,y) = \langle x-y,x-y\rangle^{1/2}$.		
		\subsection{Feller Property} Let us recall the following definition.
	\begin{definition}
		 We say that $\mm := (\mm_x)_x$ has the weak-Feller property if and only if for every sequence $x_n\to x^0\in \cH$ we have $\mm_{x_n}\rightharpoonup \mm_{x^{o}}$, \ie $\int f \dd\mm_{x_n}\to \int f \dd \mm_{x^{o}}$ for any $f\in C_{b}(\cH).$
	\end{definition}
It turns that \cref{assump:2} implies directly that the family $\mm$ is  weak-Feller.
	\begin{proposition}\label{prop:feller-property}
		Under \cref{assump:2}, $\mm$ has the weak-Feller property. Moreover, for each $x\in \cH$, $\mm_x$ has finite first-order moment.
	\end{proposition}	
		\begin{proof}
		Let $x^{o}\in\cH$ and $(x_n)_n$ a sequence of $\cH$ such that $x_n \to x^{o}$ as $n\to 0.$ Then \cref{assump:2} gives
		\[
		\Wass_{1}(\mm_{x_n},\mm_{x^{o}}) \leq \tau \Vert x_n - x^{o}\Vert,
		\] 
	and thus  $\lim_{n\to\infty}\Wass_{1}(\mm_{x_n},\mm_{x^{o}}) = 0$. Thanks to \cite[Proposition 7.1.5]{AGS}, $(\mm_{x_n})$ has uniformly integrable $p$-moments with $p\geq1$ and narrowly converges towards $ \mm_{x^{o}}$. In particular $\mm$ is weak-Feller and  each $\mm_x$ has finite first-order moments.
	\end{proof}
	\begin{remark}
		 We already know that $\mm_{x}\in\Pro$ for each $x\in\cH$ and that $x\mapsto\mm_{x}(C)$ is measurable for each $C\in\mathcal{B}$. Moreover, thanks to \cref{prop:feller-property}, we have finiteness of first-order moments of each $\mm_{x}$, so that the family $\mm$ satisfies the requirements of \cref{def:RW}. This shows that a natural setting to address dynamics of the form \eqref{eq:eq0} is the metric random walk space $[\cH,\dd,\mm]$. 
		 Many diffusion and variational problems has been studies within this framework, with allows in particular consider nonlocal continuum problems or  problems on weighted graphs (see, \eg \cite{mazon2023variational} and the references therein).
	\end{remark}
	
	\begin{remark}
	Let us point out that if $\upsilon$ is an invariant measure with respect to $\mm$ then it is also and invariant measure with respect to $\mm^{*n}$ for every $n\in\N$, where $\mm^{*n}$ is the $n$-step transition probability function given by \eqref{eq:n-step}. It turns out that weak-Feller property implies that every weak$-*$ limit $\upsilon$ of $(\mm^{*n})_{n}$ is an invariant measure of $\mm$ \cf \cite[Proposition 7.2.2]{HernandezLasserre} (see also \cite[Proposition 12.3.4]{Douc&al}). However, without assuming at first the existence of an invariant measure with respect to $\mm$, the measure $\upsilon$ may be trivial. Without further compactness assumptions on the metric space (see, \eg \cite[Theorem 7.2.3]{HernandezLasserre}) one needs some Lyapunov like condition to ensure the existence of an invariant measure $\upsilon$ of the weak-Feller family $\mm$ (see, \eg \cite[Theorem 7.2.4]{HernandezLasserre} or \cite[Theorem 12.3.3]{Douc&al}). As we will see in \cref{cor:inv-measure-OR}, another way to obtain the existence of invariant measures is having a positive lower bound on the coarse Ricci curvature of $[\cH,\dd,\mm]$.  
		\end{remark}
	
	\subsection{Ollivier-Ricci curvature}
	Let us discuss here the connexion between \cref{assump:2} and the so-called coarse or Olliver-Ricci curvature (\ORC~for short). The results can be found in \cite{OllivierY} or \cite{ollivier2010survey}. A more recent presentation can be found in \cite{mazon2023variational}.
		\begin{definition}[Ollivier-Ricci curvature \cite{OllivierY}]
		Let $[\cH,\dd,\mm]$ be a m.r.w.s. Then, for any distinct points $x,y\in\cH$, the \ORC~along $(x,y)$ is defined as:
		\begin{equation}\label{eq:OR}
			\kappa_{\mm}(x,y) = 1 - \frac{\Wass_1(\mm_x,\mm_{y})}{\dd(x,y)},
		\end{equation}
The \ORC~of $[\cH,\dd,\mm]$ is defined as
		\begin{equation}\label{eq:OR-Curv}
			\kappa_{\mm} := \inf_{x\neq y} \kappa_{\mm}(x,y).
		\end{equation} 
	\end{definition}
	
	We clearly see from \eqref{eq:OR} that $\kappa_\mm(x,y)\leq 1$. Moreover, rearranging the terms, we have 
	\begin{equation}\label{eq:W1_bound1}	
	\Wass_1(\mm_x,\mm_{y}) = \left(1-\kappa_{\mm}(x,y)\right)\dd(x,y).
	\end{equation}
	Consequently, having some lower bound $\kappa_{\mm}(x,y)\geq c \in\R$ for any $x,y\in\cH$ gives 
		\begin{equation}\label{eq:W1_bound2}		
\Wass_{1}(\mm_{x},\mm_{y})\leq (1-c)\dd(x,y),
		\end{equation}
	which describes a Lipschitz behavior of the random walk $\mm$. This has to be compared to \cref{assump:2}. Indeed, we see from \cref{assump:2} that, for $x\neq y$
	\begin{equation}\label{eq:bounds_OR}
 1-\tau\leq  \kappa_{\mm}(x,y)\leq 1,
	\end{equation}
	so according to the values of $\tau$ we have different regimes on the \ORC $~\kappa_\mm$ (\cf \cref{tab:curvature}). 
	\begin{table}[ht]
\centering
\caption{Relation between the values of $\tau$ and $\kappa_\mm$.}
\begin{tabular}[t]{|c|c|c|c|}
\hline
Values of $\tau$ &$0$  &$<1$&$\leq 1$ \\
\hline
Values of $\kappa_\mm$&$1$&$]0,1]$&$[0,1]$\\
\hline
\end{tabular}
\label{tab:curvature}
\end{table}%
	\begin{figure}
		\centering
		\includegraphics[width=0.4\textwidth]{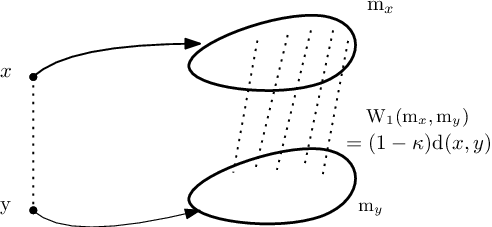}
		\caption{Illustration of the \ORC.}
		\label{fig:curvature}
	\end{figure}
	Notice that \cref{assump:2} excludes both the cases $\kappa_\mm\equiv 1$ and $\kappa_\mm<0$. Typically, $\tau = 0$, would give that $\kappa_\mm=1$ in other words $\Wass_{1}(\mm_x,\mm_y) = 0$ for any $x,y\in\cH$, \ie the distribution $\mm$ is constant.

	Moreover, it turns that there is equivalence between the lower bound on $\kappa_\mm$ in \eqref{eq:bounds_OR} and the Lipschitz behavior \eqref{eq:W1_bound2}. This is directly related to a $\Wass_1$-contraction property  \cf \cite[Proposition 20]{OllivierY}.
	\begin{proposition}\label{prop:contraction}
		Let $\mm$ be a random walk on $(\cH, \dd)$ and assume that $\mm_x$ has finite moment for all $x\in\cH$. Then
		\[
		\kappa_{\mm}(x,y)\geq c\in \R,~\forall x\neq y \iff \Wass_{1}(\nu_1\star\mm,\nu_2\star\mm)\leq (1-c)\Wass_{1}(\nu_1,\nu_2)~\forall\nu_1,\nu_2\in\Probm.
		\]
	\end{proposition}
	In view of \cref{prop:contraction}, taking $\nu_1 = \delta_x$ and $\nu_2=\delta_y$ for $x\neq y$, and $c = 1-\tau$, we get 
	\[
	\Wass_{1}(\mm_x,\mm_y) = \Wass_{1}(\delta_x\star\mm,\delta_y\star\mm) \leq \tau \Wass(\delta_x,\delta_y) = \tau \dd(x,y)
	\]
	which is exactly \cref{assump:2} since the above inequality is trivial for $x=y$.\\
	
	In the case of positive curvature, this contraction result implies the existence of a unique invariant measure for the random walk $\mm$ when the \ORC~is positive.
	\begin{corollary}[\cite{mazon2023variational}]\label{cor:postive-ORC}
		Assume that $\kappa_{\mm}(x,y)\geq c >0$ for all $x\neq y$. Then, the random walk $\mm$ has a unique invariant measure $\upsilon\in\Probm.$ Moreover, for any $\nu\in\Probm$
		\begin{enumerate}
			\item $\Wass_{1}(\nu\star\mm^{*n},\upsilon)\leq (1-c)^{n}\Wass_{1}(\nu,\upsilon)$, $\forall n\in\N$.
			\item $\Wass_{1}(\mm^{*n},\upsilon)\leq \frac{(1-c)^{n}}{c}\Wass_{1}(\delta_x,\upsilon)$, $\forall n\in\N, \forall x\in \cH$.
		\end{enumerate}
	\end{corollary}

	In our setting, the parameter regime $\tau < 1$ would define a positive \ORC. Consequently, we have the following
	\begin{corollary}\label{cor:inv-measure-OR}
		Assume that $\mm$ satisfies \cref{assump:2} with $\tau <1$. Then, there exists a unique invariant measure $\upsilon\in\Probm$ with respect to $\mm$.
	\end{corollary} 

	\section{Application: Inertial primal-dual algorithm}\label{section:application}
	 
	This section is devoted to the application of the developed results in \cref{section:firstorder} and \cref{section:hessian2ndorder} to the following class of saddle-point problems
	\begin{equation}\label{eq:application}
			\inf_{x \in \cH}\sup_{y \in \cK}~\F_{\mm^p_x}(x) + \g(x) + \langle y, \kk x \rangle - \hh(y) - \rr_{\mm^d_y} (y) ,
	\end{equation}

	where $\kk: \cH \to \cK$ is a bounded linear operator, $\cK$ is a real Hilbert space, $(\mm^{p}_{x})_{x\in\cH}$ and $(\mm^{d}_{y})_{y\in\cK}$ are two families of probability measures on $\Xi$ and $\cZ$ respectively. Recall that $\Xi$ and $\cZ$ are two Polish spaces. We equip $\cH \times \cK$ with the product space inner product and associated norlm.
	In what follows, we make the following assumptions:
	\begin{assumption}\label{assump:6}~\medskip
		\begin{enumerate}[(a)]
		\item\label[noref]{it:pd_assump_1} $\F_{\mm^p_x}(x) := \Ex_{\xi\sim \mm^p_{x}}(f(x,\xi))$, $f \in C^1(\cH \times \Xi)$, $f(x,\cdot)$ is $\mm^p_x$-measurable and $C^{1,1}_{\beta_\F}(\Xi)$ for every $x \in \cH$, and $\F_{\mm^p_{x}}(\cdot,\xi) \in C^{1,1}_{L_\F}(\cH)$ for every $\xi \in \Xi$;
		\item\label[noref]{it:pd_assump_1'} $\rr_{\mm^d_y}(y) := \Ex_{\zeta\sim \mm^d_{y}}(\rr(y,\zeta))$, with $\rr \in C^1(\cK \times \cZ)$, $\rr(y,\cdot)$ is $\mm^d_y$-measurable and $C^{1,1}_{\beta_\rr}(\cZ)$ for every $y \in \cK$, and $\rr_{\mm^d_{y}}(\cdot,\zeta) \in C^{1,1}_{L_\rr}(\cK)$ for every $\zeta \in \cZ$;
		\item\label[noref]{it:pd_assump_2} $\hh\in \Gamma_{0}(\cK)$ and $\g\in \Gamma_{0}(\cH)$;
		\item\label[noref]{it:pd_assump_3} $\F_{\mm^p_x}+\g$ is $\mu_{p}$-strongly convex for all $x \in \cH$, and $\rr_{\mm^d_y} + \hh$ is $\mu_{d}$-strongly convex for all $y \in \cK$.
		\item\label[noref]{it:pd_assump_4} There exists $\tau >0$ such that
		\begin{align*}
		\Wass_{1}(\mm^p_x \otimes \mm^d_y,\mm^p_{x'}\otimes\mm^d_{y'}) &\leq \tau \Vert (x,y) - (x',y')\Vert_{\cH \times \cK},~\mbox{for all}~x,x'\in \cH, y,y'\in \cK .
		\end{align*}
		\end{enumerate}
		$\beta_{\F},\beta_{\rr},L_{\F},L_{\rr},\mu_{p},\mu_{d}$ are all positive constants.
		 			\end{assumption}
			
\cref{assump:6}\cref{it:pd_assump_1,it:pd_assump_1',it:pd_assump_2} are to be compared to \cref{assump:1,assump:5}. This will be made clearer when we will cast the saddle point problem \eqref{eq:application} and an inclusion problem reminiscent of \eqref{eq:eq0}. 
In the same way, \cref{assump:6}\cref{it:pd_assump_3} will be sufficient for \cref{assump:3} to hold. \cref{it:pd_assump_4} is the version of \cref{assump:2} on the product space $\cH \times \cK$. \\

	Problems of the form \eqref{eq:application} arise in many fields such as image and signal processing, machine learning and partial differential equations. In the deterministic case where $\F$ and/or $\rr$ do not depend on the distributions $\mm^p_x$ and $\mm^d_y$, and where all functions are convex, various forms of such problem were studied in \eg \cite{chambolle2011first,condat2013primal, vu2013splitting,raguet&al} and many others. In \cite{bianchi2021fully}, the authors studied a stochastic variant of \eqref{eq:application} with non-state dependent measures and convex functions. Specifically, they considered $\F(x) = \Ex_{\xi\sim \mm}(f(x,\xi)), \g(x)= \Ex_{\xi\sim \mm}(g(x,\xi)), \kk = \Ex_{\xi\sim \mm}(K(\xi)), \rr\equiv 0$ and $\hh(y) = \Ex_{\xi\sim \mm}(p(x,\xi))$ for some suitable functions $f,g,h,p$ and $K$. Problems of the form \eqref{eq:application} in their full generality are difficult to tacle directly because of the presence of the (primal and dual) state-dependent distributions. To the best of our knowledge, general saddle point problems of the form \eqref{eq:application} have not been addressed in the literature.
	
\subsection{Formulation as a monotone inclusion}
As seen in \cref{section:firstorder}, the appropriate notion of solutions of \eqref{eq:application} is that of equilibria. Thus, our aim is to find an equilibrium point $(\bar{x},\bar{y})$, \ie that $(\bar{x},\bar{y})$ a solution of the static saddle point problem 
	\begin{equation}\label{eq:application_saddle_form}
			\inf_{x\in\cH} \sup_{y\in\cK}~\lag_{\mm^p_{\bar{x}},\mm^d_{\bar{y}}}(x,y),
	\end{equation}
	where we have defined the Lagrangian for the (primal and dual) probability measures $\mm^p \in \cP(\Xi)$ and $\mm^d \in \cP(\cZ)$
	\[
	\lag_{\mm^p,\mm^d}(x,y) := \F_{\mm^p}(x) + \g(x) + \langle y, \kk x \rangle - \hh(y) - \rr_{\mm^d} (y) .
	\]

If the set of saddle points is nonempty (typically under strong duality, see \eg \cite[Chapter~15]{Baucshke&Combettes}), then $(x^\star,y^\star)$ is a saddle point for \eqref{eq:application_saddle_form} if and only if 
\[
\lag_{\mm^p_{\bar{x}},\mm^d_{\bar{y}}}(x^\star,y) \leq \lag_{\mm^p_{\bar{x}},\mm^d_{\bar{y}}}(x^\star,y^\star) \leq \lag_{\mm^p_{\bar{x}},\mm^d_{\bar{y}}}(x,y^\star) \quad \forall (x,y) \in \cH \times \cK ,
\]
or equivalently, if the following optimality condition holds \cite[Corollary~19.19]{Baucshke&Combettes}\footnote{Here, only convexity is needed rather than \cref{assump:6}\cref{it:pd_assump_3}.}
	\begin{equation}\label{eq:KKT}
	\left\{
	\begin{aligned}
		 0&\in \nabla \F_{\mm^p_{\bar{x}}}(x^\star) + \partial \g(x^\star) + \kk^{*}y^\star \\
		 0&\in \nabla \rr_{\mm^d_{\bar{y}}}(y^\star) + \partial \hh(y^\star) -\kk x^\star .
	\end{aligned}
	\right.
	\end{equation}
	where $\kk^*$ is the adjoint operator of $\kk$. Therefore, $(\bar{x},\bar{y})$ will be said to be an equilibrium of \eqref{eq:application} if it is a solution of \eqref{eq:KKT}. 

	Equivalently, this can be cast as the monotone inclusion problem
	\begin{equation}\label{eq:MI-PD}
	0_{\cH\times\cK}\in \Txb(x^\star,y^\star),
	\end{equation}
	where 

	\begin{equation}\label{eq:operator_T}
	\begin{aligned}
	\bT_{\mm^p,\mm^d} &= \bA + \bL + \bB_{\mm^p,\mm^d}: \cH \times \cK \rightrightarrows \cH \times \cK, \\
	\bA &= \begin{pmatrix} \partial \g & 0 \\ 0 & \partial \hh \end{pmatrix}, \quad \bL = \begin{pmatrix} 0 & \kk^* \\ -\kk & 0 \end{pmatrix}, \quad \text{and} \quad
	\bB_{\mm^p,\mm^d} = \begin{pmatrix} \nabla \F_{\mm^p} & 0 \\ 0 & \nabla \rr_{\mm^d} \end{pmatrix} .
	\end{aligned}
	\end{equation}

\begin{remark}
The authors in \cite{wood2023stochastic} studied problems of the form
\begin{equation*}\label{eq:example-wood}
	\min_{x\in X}\max_{y\in Y} \Ex_{\xi\sim \mm_{(\bar{x},\bar{y})}} \phi(x,y,\xi),
\end{equation*}
in finite dimension where $X,Y$ are compact sets and $\phi$ is a convex-concave function that plays the role of the Lagrangian in our case. Such problems fall into the scope of \eqref{eq:application_saddle_form} where $X$ and $Y$ can be absorbed in the functions $\g$ and $\hh$ through their respective indicator functions $\iota_{X}$ and $\iota_{Y}$.
\end{remark}

The form \eqref{eq:operator_T} brings back our problem to the general setting studied in \cref{section:firstorder}. The following result shows that the operators $\bA$ and $\bB_{\mm^p,\mm^d}$ indeed verify all the appropriate properties required there. 

\begin{lemma}\label{lem:pd_application} Under \cref{assump:6}, the following properties hold for any probability measures $\mm^p \in \cP(\Xi)$ and $\mm^d \in \cP(\cZ)$:
\begin{enumerate}[(i)]
	\item\label[noref]{it:1_lem_pd} The operator $\bA$ is maximally monotone.
	\item\label[noref]{it:2_lem_pd} The operator $\bB_{\mm^p,\mm^d}$ is $\tilde{L}$-Lipschitz continuous with $\tilde{L}=\max(L_{\F},L_{\rr})$. It is maximally monotone if $f(\cdot,\xi)$ is convex for $\mm^p$-almost every $\xi \in \Xi$ and $\rr(\cdot,\zeta)$ is convex for $\mm^d$-almost every $\zeta \in \cZ$.
	\item\label[noref]{it:4_lem_pd} The operator $\bT_{\mm^p,\mm^d}$ is maximally $\tilde\mu$-strongly monotone with $\tilde\mu =  \min(\mu_p,\mu_d)$.
	\item\label[noref]{it:3_lem_pd} For any $x,x'\in\cH$ and $y,y'\in\cK$
	\[
	\sup_{(x'',y'')\in\cH\times\cK}\Vert \bB_{\mm^p_x,\mm^d_y}(x'',y'') - \bB_{\mm^p_{x'},\mm^d_{y'}}(x'',y'')\Vert_{\cH\times\cK} \leq \tilde{\beta}\tau\Vert (x,y)-(x',y')\Vert_{\cH\times\cK},
	\]
	where $\tilde{\beta}=\max(\beta_{\F},\beta_{\rr})$.
	\end{enumerate}
\end{lemma}
\begin{proof}
Thanks to \eqref{eq:operator_T} and assumption \cref{it:pd_assump_2}, \cref{it:1_lem_pd} follows from \cite[Theorem~21.2 and Proposition~20.23]{Baucshke&Combettes}. 
The operator $\bB_{\mm^p,\mm^d}$ is the gradient of the separable function $\Phi(x,y) = \F_{\mm^p}(x)+\rr_{\mm^d}(y)$. From assumptions \cref{it:pd_assump_1,it:pd_assump_1'}, we have $\nabla \F_{\mm^p}={\Ex_{\xi\sim \mm^p}}\left(\nabla f(\cdot,\xi)\right)$ and similarly for $\nabla \rr_{\mm^d}$. Thus $\Phi \in C^{1,1}_{\max(L_\F,L_\rr)}(\cH \times \cK)$ whence we get claim the first part of \cref{it:2_lem_pd}. Moreover, under convexity, $\F_{\mm^p}$ and $\rr_{\mm^d}$ are also convex, hence their gradients are maximally monotone, and we conclude using \cite[Proposition~20.23]{Baucshke&Combettes}. To show \cref{it:4_lem_pd}, observe that $\bT_{\mm^p,\mm^d}(x,y) = (\partial (\F_{\mm^p}+\g)(x),\partial (\rr_{\mm^d}+\hh)(y))+\bL(x,y)$. The first part is maximally strongly monotone by \cref{it:pd_assump_3} and \cite[Example~22.4, Theorem~21.2 and Proposition~20.23]{Baucshke&Combettes}. In addition, $\bL$ is a skew-symmetric linear operator, hence maximally monotone (see \cite[Example~20.35]{Baucshke&Combettes}). We conclude using \cite[Lemma~2.4]{brezismaxmon}. 

		The proof of \cref{it:3_lem_pd} is similar to that of \cref{corollary:1}. Indeed, let us fix $(x'',y'')\in\cH\times\cK$. We then have, for any $x,y\in\cH$ and $y,y'\in\cK$
			\[
		\begin{aligned}
			&\Vert \bB_{\mm^p_x,\mm^d_y}(x'',y'') - \bB_{\mm^p_{x'},\mm^d_{y'}}(x'',y'') \Vert_{\cH\times\cK} \\
			&= \sqrt{\Big\Vert \nabla \F_{\mm^p_x}(x'') - \nabla \F_{\mm^p_{x'}}(x'') \Big\Vert_{\cH}^2 
			+ \Big\Vert \nabla \rr_{\mm^d_y}(y'') - \nabla \rr_{\mm^d_{y'}}(y'') \Big\Vert_{\cK}^2} \\
			&\leq \sqrt{\beta_{\F}^2\Wass_1(\mm^p_x,\mm^p_{x'})^2 + \beta_{\rr}^2\Wass_1(\mm^d_y,\mm^d_{y'})^2} \\
			&\leq \max(\beta_{\F},\beta_{\rr})\Wass_1(\mm^p_x \otimes \mm^p_{x'},\mm^d_y \otimes \mm^d_{y'}) \\
			&\leq \max(\beta_{\F},\beta_{\rr})\tau\Vert (x,y) - (x',y')\Vert_{\cH \times \cK} .
		\end{aligned}
		\]
\end{proof}

\subsection{Existence and uniqueness of equilibrium}
We are now in a position to prove the existence and uniqueness of an equilibrium to the problem \eqref{eq:application_saddle_form}
	\begin{theorem}[Existence and uniqueness of equilibrium point]\label{thm:equilibrium-pd }
		Under \cref{assump:6}, the map
		\[
		S:(x,y)\in\cH\times\cK\mapsto \zer(\bT_{\mm^p_x,\mm^d_y})=\{(u,v)\in\cH \times \cK:~(0,0)_{\cH\times\cK}\in \bT_{\mm^p_x,\mm^d_y}(u,v)\}
		\]
is $\tilde{\rho}$-Lipschitz with  
\begin{equation}\label{eq:pdrhomu}
\tilde\rho := \frac{\tau\tilde{\beta}}{\tilde\mu} ,
\end{equation}
where we recall that $\tilde\mu = \min(\mu_p,\mu_d)$ and $\tilde{\beta}=\max(\beta_{\F},\beta_{\rr})$.
In particular, if $\tilde\rho <1$, there is a unique equilibrium point $(\bar x,\bar y)$.
	\end{theorem}
	An immediate consequence of this result is that \eqref{eq:KKT} has a unique solution which is $(\bar x,\bar y)$.
	\begin{proof}
		Following the same lines as in \cref{thm:equilibrium} and using \cref{lem:pd_application}\cref{it:4_lem_pd}, we conclude. 
			\end{proof}
\begin{remark}
\cref{thm:equilibrium-pd } is to be compared to \cite[Theorem 2.6]{wood2023stochastic}. Notice that the monotone inclusion \eqref{eq:MI-PD} defining the equilibrium in \cite{wood2023stochastic}, handled through a a variational inequality there to the presence of constraints on $x$ and $y$, can be cast in our setting through normal cones absorbed in $\partial \g$ and $\partial \hh$ under appropriate qualification conditions.
\end{remark}			
			
Following the reasoning of \cref{section:firstorder,section:hessian2ndorder}, let us define the following gap function
\begin{equation}
	\EE_{\bar{x},\bar{y}}(x,y) = \bB_{\mm^p_x,\mm^d_{y}}(x,y) - \bB_{\mm^p_{\bar x},\mm^d_{\bar{y}}}(x,y).
\end{equation}
Mimicking the proof of \cref{lemma:3} using \cref{assump:6},  we prove the following.
\begin{lemma}\label{lemma:Lipschitz_gap}
Under \cref{assump:6}, $\EE_{\bar{x},\bar{y}}(.)$ is $(2\tilde{L}+\tau\tilde{\beta})$-Lipschitz continuous, where $\tilde{L} = \max(L_\F,L_\rr)$.
\end{lemma}

\subsection{Related first and second-order dynamics}
 To lighten the notation, from now on we set $Z(t) := (x(t),y(t))$, $\bar{Z} := (\bar{x},\bar{y})$ is the equilibrium of \eqref{eq:application_saddle_form} (see \cref{thm:equilibrium-pd }), and $\mm_{\bar{Z}} := \mm^p_{\bar{x}} \otimes \mm^d_{\bar{y}}$.

\subsubsection{First order system}
Given an initial data $Z(t_0) = (x(t_0),y(t_0))\in \dom(g)\times\dom(h)$, we consider the following first-order system associated to the monotone inclusion \eqref{eq:MI-PD}: 
	\begin{equation}\tag{SPDS}\label{eq:spds}
		\dot{Z}(t) +  \Tzb(Z(t)) +  \EE_{\bar{Z}}(Z(t)) \ni 0_{\cH\times\cK}	.
	\end{equation}

	Arguing as in \cref{prop:1st_order_wp}\footnote{ Here, we invoke \cite[Proposition~3.12]{brezismaxmon} rather than \cite[Proposition~3.13]{brezismaxmon} as $\Tzb$ is a subdifferential operator perturbed by a bounded linear operator. This explains why we do not need any domain interiority assumption in this case.} and \cref{thm:1st_order_convergence} we have the following result.
	\begin{proposition}
	Assume that \cref{assump:6} holds and $\tilde{\rho} < 1$. Then, for any initial data $Z(t_0) = (x(t_0),y(t_0))\in \dom(g)\times\dom(h)$, \eqref{eq:spds} admits a unique strong solution $Z:t\in [t_0,+\infty[ \; \mapsto (x(t),y(t))$. Moreover, 
		\[
		 \Vert Z(t) - \bar{Z}\Vert_{\cH \times \cK} \leq C e^{-2\tilde{\mu}(1 - \tilde{\rho})t},~\forall t\geq t_0,
		\]
		with $C =  \Vert Z_0 - \bar{Z}\Vert_{\cH \times \cK} e^{2\tilde{\mu}(1 - \tilde{\rho})t_0}$, where $\tilde{\rho}$ and $\tilde{\mu}$ are defined in \eqref{eq:pdrhomu}.
	\end{proposition}

	\subsubsection{Second-order system}
	 
	As we have done in \cref{section:hessian2ndorder}, here, we restrict ourselves to the smooth setting where both $\g$ and $\hh$ are smooth with Lipschitz continuous gradient.
 
In this case, the operator $\bA$, hence $\Tzb$, is single-valued, Lipschitz continuous and strongly monotone. In explicit forms, for $Z=(x,y)$, we have
		\begin{equation}\label{eq:operator_T_smooth}
		\Tzb(Z)  = \left(\nabla_x \lag_{\mm^p_{\bar{x}},\mm^d_{\bar{y}}}(x,y), - \nabla_y \lag_{\mm^p_{\bar{x}},\mm^d_{\bar{y}}}(x,y)\right).
	\end{equation}

	 We propose the following inertial system associated to \eqref{eq:MI-PD}
	\begin{equation}\tag{ISPDS$_{2\sqrt{\tilde{\mu}}}$}\label{eq:sipds}
	\begin{aligned}
	\ddot{Z}(t) + 2\sqrt{\tilde\mu}\dot{Z}(t) + 
	\begin{pmatrix}
	\nabla_x \lag_{\mm^p_{\bar{x}},\mm^d_{\bar{y}}}\left(x(t),y(t)+\frac{1}{\sqrt{\tilde\mu}}\dot{y}(t)\right) \\ 
	- \nabla_y \lag_{\mm^p_{\bar{x}},\mm^d_{\bar{y}}}\left(x(t)+\frac{1}{\sqrt{\tilde\mu}}\dot{x}(t),y(t)\right)
	\end{pmatrix} + \EE_{\bar{Z}}(Z(t)) = 0 .
	\end{aligned}
	\end{equation}
	Unperturbed second order primal-dual dynamical systems for solving saddle point problems have been actively studied in recent years, essentially with vanishing viscous damping; see e.g., \cite{Attouch_admm} and \cite{He24} and references therein.

Unlike \eqref{eq:sec2_explicit_st_cvx}, we did not include geometric damping in \eqref{eq:sipds} for the sake of simplicity. Moreover, one can see that the system \eqref{eq:sipds} is driven by the gradient of the Lagrangian but evaluated at extrapolated points, which is standard approach for primal-dual second-order systems. In fact, one has to keep in mind that unlike \eqref{eq:sec2_explicit_st_cvx}, the system \eqref{eq:sipds} is driven by an operator deriving from the Lagrangian which is convex-concave and contains a bilinear form. This is the reason underlying the inclusion of the extrapolation step. This will also necessitate to modify the Lyapunov analysis in the proof \cref{thm:th1} though the main ingredients will remain essentially the same. \\

We first state that \eqref{eq:sipds} is well posed. This follows from the same arguments as \cite[Theorem~5]{Attouch_admm}, using the Cauchy-Lipschitz theorem since $\nabla_x \lag_{\mm^p_{\bar{x}},\mm^d_{\bar{y}}}$ and $\nabla_y \lag_{\mm^p_{\bar{x}},\mm^d_{\bar{y}}}$ are globally Lipschitz continuous by assumption, and so is $\EE_{\bar{Z}}$ by \cref{lemma:Lipschitz_gap}.
\begin{proposition}\label{prop:wellglobalsipds}
Suppose that \cref{assump:6}\cref{it:pd_assump_1,it:pd_assump_2,it:pd_assump_4} hold and, moreover, that $\g$ and $\hh$ are also continuously differentiable with Lipschitz continuous gradient. Then, for any given initial condition $(x(t_0), \dot{x}(t_0))=(x_0,u_0)\in \cH \times \cH$ and $(y(t_0), \dot{y}(t_0))=(y_0,v_0)\in \cK \times \cK$ the evolution system \eqref{eq:sipds} has a unique strong global solution $Z(\cdot)=(x(\cdot),y(\cdot))$ with 
\begin{itemize}
\item $Z \in C^1([0,+\infty[;\cH \times \cK)$;
\item $Z$ and $\dot{Z}$ are absolutely continuous on every compact subset of the interior of $[t_0,+\infty[$ (hence almost everywhere differentiable);
\item for almost all $t \in [t_0,+\infty[$, \eqref{eq:sipds} holds with $Z(t_0) = (x_0,y_0)$ and $\dot{Z}(t_0) = (u_0,v_0)$.
\end{itemize}
\end{proposition}

\medskip

Let us now move to the convergence properties of \eqref{eq:sipds}. We define the following energy function $\V:[t_0,\infty[\to\R^+$ by
	\[
	\V(t):= \lag_{\mm^p_{\bar{x}},\mm^d_{\bar{y}}}(x(t),\bar{y}) - \lag_{\mm^p_{\bar{x}},\mm^d_{\bar{y}}}(\bar{x},y(t)) + \frac{1}{2}\Vert v(t)\Vert_{\cH \times \cK}^{2}
	\]
	with 
	\[
	v(t) := \sqrt{\tilde{\mu}}\left(Z(t) - \bar{Z}\right) + \dot{Z}(t),
	\]
	Observe that the Lagrangian gap in $\V$ is nonnegative and convex in $Z(t)$.

\begin{theorem}\label{thm:th2}
		 Assume that \cref{assump:6} holds and, moreover, that $\g$ and $\hh$ are also continuously differentiable with Lipschitz continuous gradient. Let $t\in[t_0,\infty[\mapsto(x(t),y(t))$ be the solution of \eqref{eq:sipds}. Suppose that  $\tilde{\rho}< \frac{\sqrt{2}}{4}$, where $\tilde{\rho}$ is as given in \eqref{eq:pdrhomu}.
		
		We then have for all $t\geq t_0$:
	\begin{equation}\label{eq:thm1_lag_est}
			 \frac{\tilde{\mu}}{2}\Vert  Z(t)-\bar{Z}\Vert_{\cH \times \cK}^2 \leq \lag_{\mm^p_{\bar{x}},\mm^d_{\bar{y}}}(x(t),\bar{y}) - \lag_{\mm^p_{\bar{x}},\mm^d_{\bar{y}}}(\bar{x},y(t)) \leq \V(t_0) e^{-\frac{\sqrt{\tilde{\mu}}}{4}(t-t_0)}.
	\end{equation}
	
	\end{theorem}
	\begin{proof}
	 To lighten notation, we drop dependence of the norm and inner product on the underlying space which is to be understood from the context. At first, we have, using \eqref{eq:sipds}
	 
	\begin{align*}
	\dot{v}(t) 
	= \sqrt{\tilde\mu}\dot{Z}(t) + \ddot{Z}(t)
	&= -\sqrt{\tilde\mu}\dot{Z}(t) - 
	\begin{pmatrix}
	\nabla_x \lag_{\mm^p_{\bar{x}},\mm^d_{\bar{y}}}\left(x(t),y(t)+\frac{1}{\sqrt{\tilde\mu}}\dot{y}(t)\right) \\ 
	- \nabla_y \lag_{\mm^p_{\bar{x}},\mm^d_{\bar{y}}}\left(x(t)+\frac{1}{\sqrt{\tilde\mu}}\dot{x}(t),y(t)\right)
	\end{pmatrix} - \EE_{\bar{Z}}(Z(t)) \\
	&= -\sqrt{\tilde\mu}\dot{Z}(t) - 
	\begin{pmatrix}
	\nabla (\F_{\mm^p_{\bar{x}}}+\g)(x(t)) \\ 
	\nabla (\rr_{\mm^d_{\bar{y}}}+\hh)(y(t))
	\end{pmatrix} - \bL \left(Z(t)+\frac{1}{\sqrt{\tilde\mu}}\dot{Z}(t)\right) - \EE_{\bar{Z}}(Z(t)) \\
	&= -\sqrt{\tilde\mu}\dot{Z}(t) - \Tzb(Z(t)) - \frac{1}{\sqrt{\tilde\mu}}\bL\dot{Z}(t) - \EE_{\bar{Z}}(Z(t)) .
	\end{align*}
    We thus get,
		\begin{align*}
		\dot{\V}(t) 
		&= \langle\left(\nabla_{x} \lag_{\mm^p_{\bar{x}},\mm^d_{\bar{y}}}(x(t),\bar{y}), -\nabla_{y} \lag_{\mm^p_{\bar{x}},\mm^d_{\bar{y}}}(\bar{x},y(t))\right),\dot{Z}(t)\rangle + \langle v(t),\dot{v}(t)\rangle \\
		&= \langle \left(\nabla (\F_{\mm^p_{\bar{x}}}+\g)(x(t)),\nabla (\rr_{\mm^d_{\bar{y}}}+\hh)(y(t))\right)+\bL\bar{Z},\dot{Z}(t)\rangle + \langle v(t),\dot{v}(t)\rangle \\
		&= \langle \Tzb(Z(t)) + \bL(\bar{Z}-Z(t)),\dot{Z}(t)\rangle \\
		&~ + \langle \sqrt{\tilde{\mu}}\left(Z(t) - \bar{Z}\right) + \dot{Z}(t),-\sqrt{\tilde\mu}\dot{Z}(t) - \Tzb(Z(t)) - \frac{1}{\sqrt{\tilde\mu}}\bL\dot{Z}(t) - \EE_{\bar{Z}}(Z(t))\rangle .
		\end{align*}
		After some simplifications and using that $\bL$ is skew-symmetric, we arrive at
		\begin{equation}\label{eq:der_pdV}
		\dot{\V}(t) + \sqrt{\tilde{\mu}}\langle \Tzb(Z(t)),Z(t)-\bar{Z}\rangle + \tilde{\mu} \langle\dot{Z}(t),Z(t)-\bar{Z}\rangle + \sqrt{\tilde{\mu}}\Vert\dot{Z}(t)\Vert^2 = - \langle v(t),\EE_{\bar{Z}}(Z(t))\rangle.
		\end{equation}

		Rather than using $\tilde{\mu}$-strong monotonicity of $\Tzb$ (\cref{lem:pd_application}\cref{it:4_lem_pd}), we will now prove an alternative bound on the first inner product in \eqref{eq:der_pdV} that will be useful for our Lyapunov analysis. We have
		\begin{align*}
		&\langle \Tzb(Z(t)),Z(t)-\bar{Z}\rangle \\
		&= \langle \nabla (\F_{\mm^p_{\bar{x}}}+\g)(x(t)),x(t)-\bar{x} \rangle + \langle \nabla (\rr_{\mm^d_{\bar{y}}}+\hh)(y(t)),y(t)-\bar{y} \rangle + \langle \bL Z(t),Z(t)-\bar{Z} \rangle \\
		&= \langle \nabla (\F_{\mm^p_{\bar{x}}}+\g)(x(t)),x(t)-\bar{x} \rangle + \langle \nabla (\rr_{\mm^d_{\bar{y}}}+\hh)(y(t)),y(t)-\bar{y} \rangle - \langle \kk^* y(t),\bar{x}\rangle + \langle \kk x(t),\bar{y}\rangle .
		\end{align*}
		Now, by \cref{assump:6}\cref{it:pd_assump_3}, we have 
		\begin{align*}
		\langle \nabla (\F_{\mm^p_{\bar{x}}}+\g)(x(t)),x(t)-\bar{x} \rangle 
		&\geq (\F_{\mm^p_{\bar{x}}}+\g)(x(t)) - (\F_{\mm^p_{\bar{x}}}+\g)(\bar{x}) + \frac{\mu_p}{2}\|x(t)-\bar{x}\|^2 \\
		\langle \nabla (\rr_{\mm^d_{\bar{y}}}+\hh)(y(t)),y(t)-\bar{y} \rangle 
		&\geq (\rr_{\mm^d_{\bar{y}}}+\hh)(y(t)) - (\rr_{\mm^d_{\bar{y}}}+\hh)(\bar{y}) + \frac{\mu_d}{2}\|y(t)-\bar{y}\|^2 .
		\end{align*}
		Summing we get that
		\begin{align*}
		&\langle \Tzb(Z(t)),Z(t)-\bar{Z}\rangle \\
		&\geq (\F_{\mm^p_{\bar{x}}}+\g)(x(t)) - (\rr_{\mm^d_{\bar{y}}}+\hh)(\bar{y}) + (\rr_{\mm^d_{\bar{y}}}+\hh)(y(t)) - (\F_{\mm^p_{\bar{x}}}+\g)(\bar{x}) - \langle \kk^* y(t),\bar{x}\rangle + \langle \kk x(t),\bar{y}\rangle + \frac{\tilde{\mu}}{2}\|Z(t)-\bar{Z}\|^2 \\
		&= \lag_{\mm^p_{\bar{x}},\mm^d_{\bar{y}}}(x(t),\bar{y}) - \lag_{\mm^p_{\bar{x}},\mm^d_{\bar{y}}}(\bar{x},y(t)) + \frac{\tilde{\mu}}{2}\|Z(t)-\bar{Z}\|^2 .
		\end{align*}
		Plugging this into \eqref{eq:der_pdV}, we get
		\begin{equation}\label{eq:thm2_V1}
		\begin{aligned}
		&\dot{\V}(t) + \sqrt{\tilde{\mu}} \left(\lag_{\mm^{p}_{\bar{x}},\mm^{d}_{\bar{y}}}(x(t),\bar{y}) - \lag_{\mm^p_{\bar{x}},\mm^d_{\bar{y}}}(\bar{x},y(t)) + \frac{\tilde{\mu}}{2}\|Z(t)-\bar{Z}\|^2 +  \sqrt{\tilde{\mu}} \langle\dot{Z}(t),Z(t)-\bar{Z}\rangle + \Vert\dot{Z}(t)\Vert^2\right) \\
		&= \dot{\V}(t) + \sqrt{\tilde{\mu}} \left(\lag_{\mm^p_{\bar{x}},\mm^d_{\bar{y}}}(x(t),\bar{y}) - \lag_{\mm^p_{\bar{x}},\mm^d_{\bar{y}}}(\bar{x},y(t)) + \frac{1}{2}\|\sqrt{\tilde{\mu}}(Z(t)-\bar{Z})+\dot{Z}(t)\|^2\right) + \frac{\tilde{\mu}}{2}\|\dot{Z}(t)\|^2 \\
		&= \dot{\V}(t) + \sqrt{\tilde{\mu}}\V(t) + \frac{\sqrt{\tilde{\mu}}}{2}\|\dot{Z}(t)\|^2
		\leq  \Vert v(t)\Vert\Vert\EE_{\bar{Z}}(Z(t))\Vert,
		\end{aligned}
		\end{equation}
		where we used Cauchy-Schwarz inequality in the last line.
		Arguing as above, using again \cref{assump:6}\cref{it:pd_assump_3}, we have
		\begin{equation}\label{eq:Vlagbnd}
		\V(t) \geq \lag_{\mm^p_{\bar{x}},\mm^d_{\bar{y}}}(x(t),\bar{y}) - \lag_{\mm^p_{\bar{x}},\mm^d_{\bar{y}}}(\bar{x},y(t)) \geq \frac{\tilde{\mu}}{2}\|Z(t)-\bar{Z}\|^2 .
		\end{equation}
		Using this in \eqref{eq:thm2_V1} yields the bound
		\begin{equation}\label{eq:thm2_V2}
			\dot{\V}(t) + \frac{\sqrt{\tilde{\mu}}}{2} \V(t) + \frac{\sqrt{\tilde{\mu}}}{2} \Vert\dot{Z}(t)\Vert^2 + \frac{\tilde{\mu}^{3/2}}{4}\Vert Z(t)-\bar{Z}\Vert^2 \leq  \Vert v(t)\Vert\Vert\EE_{\bar{Z}}(Z(t))\Vert.
		\end{equation}	
		We have, using Young's inequality, \cref{lem:pd_application}\cref{it:3_lem_pd} and the fact that $\V(t)\geq \frac{1}{2}\Vert v(t)\Vert^2$
		\[
		\begin{aligned}
		\Vert v(t)\Vert\Vert\EE_{\bar{Z}}(Z(t))\Vert \leq \frac{\sqrt{\tilde{\mu}}}{8}\Vert  v(t)\Vert^2 + \frac{2}{\sqrt{\tilde{\mu}}}\Vert\EE_{\bar{Z}}(Z(t))\Vert^2 \leq \frac{\sqrt{\tilde{\mu}}}{4}\V(t) + \frac{2\tilde{\beta}^{2}\tau^2}{\sqrt{\tilde{\mu}}}\Vert Z(t)-\bar{Z}\Vert^2. \\
		\end{aligned}
		\]
		Inserting into \eqref{eq:thm2_V2} and rearranging gives
		\[
		\dot{\V}(t)+\frac{\sqrt{\tilde{\mu}}}{4}\V(t) + \frac{\sqrt{\tilde{\mu}}}{2}\Vert\dot{Z}(t)\Vert^2 + \sqrt{\tilde{\mu}}\left(\frac{\tilde{\mu}}{4} - \frac{2\tilde{\beta}^{2}\tau^2}{\tilde{\mu}}\right)\Vert Z(t)-\bar{Z}\Vert^2\leq 0.
		\]
		Since $\frac{\tilde{\beta}\tau}{\tilde{\mu}}<\frac{\sqrt{2}}{4}$ by assumption, we have
		\[
		\dot{\V}(t)+\frac{\sqrt{\tilde{\mu}}}{4}\V(t) \leq 0,
		\]
		which gives after integration yields
		\begin{equation}\label{eq:thm2_estimate_V}
			\V(t)\leq \V(t_0) e^{-\frac{\sqrt{\tilde{\mu}}}{4}(t-t_0)},~\forall t\geq t_0.
		\end{equation}
		Plugging this into \eqref{eq:Vlagbnd}, we prove the claim.

		\end{proof}

	\section{Comments, extensions and future work}\label{section:conclusion}
	In this paper, we adopted a dynamical system approach to study some stochastic optimization problems with state-dependent distributions. We investigated the existence and uniqueness of equilibrium points, well-posedness as well as convergence properties of the trajectories, for both first and second-order dynamics. We highlighted some dynamical and geometrical properties of the state-dependent distributions suggesting that the natural framework to study problems of the form \eqref{eq:inclusion_intro} is the one of \textit{metric random walk spaces}. More particularly, the notion of coarse Ricci curvature gives a new insight on the geometrical hidden structure of this kind of problems. Finally, we discussed as an application the inertial primal-dual algorithm. We present here some ongoing works, possible extensions as well as some open problems.
	\subsection*{Inertial algorithms}
	Relying on the discretization of the dynamics studied in \cref{section:hessian2ndorder} and \cref{section:application}, more specifically, \eqref{eq:sec2_explicit_st_cvx} and \eqref{eq:sipds}	, we obtain new inertial algorithms with Hessian-driven damping for stochastic optimization problems with decision-dependent distributions. These algorithms exhibit rapid convergence properties and can also be adapted to the nonsmooth case. This is being addressed in ongoing work.	
	\subsection*{Implicit Hessian damping}
We focused in \cref{section:hessian2ndorder} on the explicit Hessian damping in the smooth case. Yet, it is possible to consider implicit damping as in \cite{Attouch-Chbani-Fadili-Riahi,Attouch-Fadili-Kungertsev}
			\begin{equation}
				\tag{ISIHD$_{\mm,\gamma}$}
		\label{eq:sec2_implicit}
		\ddot{x}(t) + \gamma(t)\dot{x}(t) + \nabla G_{\mm_{\bar{x}}}\left(x(t) +\omega \dot{x}(t)\right)  + \e_{\bar{x}}(x(t))  = 0,
	\end{equation}
 The dynamics \eqref{eq:sec2_implicit} is referred to as an Inertial System with  Implicit Hessian damping, since one can observe, using Taylor expansion
	\[
	\nabla G_{\mm_{\bar{x}}}\left(x(t) +\omega \dot{x}(t)\right)\approx\nabla G_{\mm_{\bar{x}}}x(t) +\nabla^2 G_{\mm_{\bar{x}}}(x(t)) \dot{x}(t).
	\]
	As it was observed in \cite{Attouch-Fadili-Kungertsev}, higher-order moments of the perturbation $\e_{\bar{x}}$ are required to get fast convergence  guarantees in the implicit case compared to the explicit one. Since in our analysis (see \cref{thm:th1}) no integrability assumption on $\e_{\bar{x}}$ is needed, it is interesting to investigate the effect of implicit Hessian damping, both in the smooth and nonsmooth cases. 
		\subsection*{Tikhonov-regularization}
		In \cref{section:hessian2ndorder} we restricted ourselves, for sake of simplicity, the analysis to the case where the operator $F_{\mm_{\bar{x}}}$ is smooth. However, it is possible to consider second-order dynamics for general (and possibly nonpotential) operators, by considering, for $\lambda>0$ the following dynamic
					\begin{equation}
				\tag{ISEHD$_{\lambda,\gamma}$}
					\label{eq:yosida_reg_hessian}
				\ddot{x}(t) + \gamma(t)\dot{x}(t) + F^{\lambda(t)}_{\mm_{\bar{x}(t)}}(x(t)) + \omega\frac{\dd}{\dd t}\left(F^{\lambda(t)}_{\mm_{\bar{x}(t)}}(x(t))\right) + \e_{\bar{x}}(x(t))= 0.
			\end{equation}
			where $F_{\mm_{\bar{x}}}^{\lambda}$ is the so-called Yosida approximation of $F_{\mm_{\bar{x}}}$ defined by $F_{\mm_{\bar{x}}}^{\lambda}= \frac{1}{\lambda}\left(\id - J_{\lambda F_{\mm_{\bar{x}}}}\right)$ and $J_{\lambda F_{\mm_{\bar{x}}}} = (\id + JF_{\mm_{\bar{x}}})^{-1}$ is the resolvent of $ F_{\mm_{\bar{x}}}$. This approach comes with several advantages. First, the Yosida approximation is single-valued so that  there is no nonsmoothness to take care of.
In addition one can exploit the $\lambda-$cocoercivity of $F_{\mm_{\bar{x}}}^{\lambda}$ and the fact that $\zer F_{\mm_{\bar{x}}} = \zer F_{\mm_{\bar{x}}}^{\lambda}$. The approach was used in \cite{Attouch&Peypouquet} for $\omega=0$ and in the recent work \cite{Attouch&Laszlo} for Newton-like dynamics. We are exploring the adaptation of this techniques to stochastic monotone inclusions with state-dependent distribution in an ongoing work. 
		\subsection*{Weaker Assumptions}
We have seen that one of the crucial assumptions in the analysis is \cref{assump:2}, which concerns the Lipschitz behavior of the distribution $(\mm_x)_{x}$. A natural question that arises is what happens under a weaker assumption. For example, when $x\mapsto\mm_x$ is Hölder continuous. We are not aware of any existing results in this direction.  We plan to investigate this in future work.
\section*{Acknowledgement}
The work of H.E was supported by the ANR grant, reference: ANR-20-CE92-0037.

\begin{appendices}
		
\section{Gronwall inequalities}
					In this section we list several auxiliary results that we make use of in the paper.
			\begin{lemma}[Gronwall's lemma: differential form]\label{lemma:gronwall1}
				Let $u,v$ be two $C^{0}$ (resp. $C^1$) nonnegative function on $[0,T]$ and let $w$ be a continuous function on $[0,T]$. We assume that
				\begin{equation}
					\frac{1}{2}\frac{\dd}{\dd t} u^{2}(t) \leq w(t)u^{2}(t) + u(t) v(t)~~\mbox{on}~~(0,T),
				\end{equation}
				then, for any $t\in [0,T]$
				\begin{equation}
					u(t)\leq u(0) e^{K(t)} + \int_{0}^{t}v(s)e^{K(t)-K(s)}\dd s,
				\end{equation}
				where $K(t) = \int_{0}^{t} w(s) \dd s$.
			\end{lemma}
			\begin{lemma}\cite[Lemma A.5]{brezismaxmon}\label{lem:gronwall2}
				Let $v\in L^{1}(t_0,T;\R^+)$ and $u\in C^{0}(t_0,T)$ such that
				\[
				\frac{1}{2}u^{2}(t) \leq \frac{1}{2}	c^2 + \int_{t_0}^{t} u(s)v(s)\dd s,
				\]
				for some $c\geq 0$ for all $t\in [t_0,T]$. Then
				\[
				\vert u(t)\vert \leq c + \int_{t_0}^{t} v(s)\dd s.
				\]
			\end{lemma}
			
			\section{Banach Fixed point theorem $\&$ Picard iterative method}\label{appendix:fixed-point}
			\begin{theorem}[see, \eg \cite{ABG,brezisFA}]\label{thm:fixedpoint}
				Let $(\X,\dd)$ be a complete metric space and $S:\X\to\X$ be a strict contraction, \ie there exists a constant $\varrho<1$ such that
				\[
				\dd(S(x),S(y))\leq \varrho \dd(x,y), \forall x,y\in\X.
				\]
				Then, there exists a unique $\bar{x}\in\X$ such that $S(\bar{x}) = \bar{x}$. Moreover, for any $x_0\in\X$, the sequence starting from $x_0$ with $x_{n+1} = S(x_n)$ for all $n\in\nats$ converges to $\bar{x}$ as $n$ goes to $\infty$.
			\end{theorem}

		\end{appendices}

\bibliographystyle{abbrv}
\bibliography{biblio}

\end{document}